\documentclass[12pt,a4paper]{article}
\usepackage{psfrag}
\usepackage{graphicx}
\usepackage{color}
\usepackage{latexsym}
\usepackage{amsmath,amsthm,amsfonts}

\newcommand\vectdue[2]{\left(\begin{matrix}#1\cr #2\end{matrix}\right)}
\newcommand\R{\mathbb{R}}
\newcommand\Z{\mathbb{Z}}
\newcommand\N{\mathbb{N}}
\newcommand\calH{\mathcal{H}}
\newcommand\calL{\mathcal{L}}

\newcommand\e{\varepsilon}
\newcommand\eps{\varepsilon}
\newcommand\supp{\mathop{\mathrm{supp}}}

\newcommand\dist{{\mathrm{dist}}}
\newcommand\rel{{\mathrm{rel}}}

\newcommand\Id{{\mathrm{Id}}}
\newcommand\Tr{{\mathrm{Tr}}}
\newcommand\rank{{\mathrm{rank\,}}}

\newcommand\loc{{\mathrm{loc}}}
\newcommand\diam{\mathop{\mathrm{diam}}}
\newtheorem{theorem}{Theorem}[section]
\numberwithin{equation}{section}
\newtheorem{lemma}[theorem]{Lemma}
\newtheorem{remark}[theorem]{Remark}
\newtheorem{proposition}[theorem]{Proposition}

\begin{document}
\begin{center}
{ \LARGE
Singular kernels, multiscale decomposition of 
microstructure, and dislocation models \\[5mm]}
%
%Multiscale decomposition of microstructure for 
%dislocation models with multiple slip
{\today}\\[5mm]
Sergio Conti$^{1}$, Adriana Garroni$^{2}$, and Stefan M\"uller$^{1,3}$\\[2mm]
{\em $^{1}$ 
Institut f\"ur Angewandte Mathematik, 
Universit\"at Bonn\\ Endenicher Allee 60, 53115 Bonn,
Germany}\\[1mm]
{\em $^{2}$ Dipartimento di Matematica, Sapienza, Universit\`a di Roma\\
P.le A. Moro 2, 00185 Roma, Italy}\\
[1mm]
{\em $^{3}$ 
Hausdorff Center for Mathematics, Universit\"at Bonn\\
Endenicher Allee 60, 53115 Bonn,
Germany}
\\[3mm]
\begin{minipage}[c]{0.8\textwidth}
We consider a model for dislocations in crystals introduced 
by Koslowski, Cuiti{\~n}o and Ortiz, which includes elastic
interactions via a singular kernel behaving as the $H^{1/2}$ 
norm of the slip. We obtain a sharp-interface limit
of the model 
within the framework of $\Gamma$-convergence. 
From an analytical point of view, our functional is a vector-valued
generalization of the one studied by Alberti, Bouchitt\'e and Seppecher to
which their rearrangement argument no longer applies. 
Instead we show that the microstructure
must be approximately one-dimensional on most length scales and exploit this
property to derive a sharp lower bound.
\end{minipage}
\end{center}

\section{Introduction}
\subsection{The result}\label{result}

We consider the functional
\begin{alignat}{1}\label{uno}
E_\e^*[u,\Omega]=& \int_{\Omega\times\Omega}
\sum_{i,j=1}^N
\Gamma_{ij}(x-y)
\left[ u_i(x)-u_i(y)\right] \left[ u_j(x)-u_j(y)\right] dxdy
\nonumber \\
&
+\frac1\e\int_\Omega\dist^2(u(x), \Z^N) dx 
\end{alignat}
where $\Omega\subset\R^2$ and $\Gamma$ is a matrix-valued kernel scaling as
$|x-y|^{-3}$, i.e., the first term is bounded from 
above and below by a multiple of the $H^{1/2}$ norm. 

This functional arises in the study of  phase field models for dislocations
(see Section~\ref{dislocations} below). 
Its main feature is that it contains two competing terms: a nonconvex term which favours
integer values of the vector-valued phase field $u$, and a regularizing term. 
This is an example of a large class of problems which share this structure, 
the classical example being the well-known Cahn-Hilliard model from the
gradient-theory of fluid-fluid phase transitions, which contains a two-well
potential depending on a scalar phase field, and a local regularization given
by the Dirichlet integral. The analysis of the asymptotic behavior in terms of
$\Gamma$-convergence for this functional goes back to Modica and Mortola
\cite{ModicaMortola77}, see also \cite{Modica87,Sternberg88}. 
Generalizations to
multiwell problems, to vector-valued problems, and to anisotropic
regularizations have been studied by several authors
\cite{Baldo90,Bouchitte90,FonsecaTartar89,BarrosoFonseca94}. 
All these problems give rise in the limit to a sharp-interface model
characterized by a line-tension energy density.
The local character of the regularization leads to a scaling
property that permits to identify the line-tension energy density through a
cell-problem formula. 

The functional (\ref{uno}) is substantially more challenging since the
regularization via the Dirichlet integral is replaced by a singular nonlocal
term, which behaves as the $H^{1/2}$ norm. The (logarithmic) failure of the embedding of $H^{1/2}$
into continuous functions reflects the fact that all length scales play
a role and that the appropriate rescaling is logarithmic. This eliminates the
possibility to select one dominant length scale and to focus on a cell problem
on that scale. An additional difficulty lies in the fact that 
(\ref{uno}) is a vectorial problem, anisotropic, and that
the lower-order term has infinitely many minima. 

In the scalar, isotropic case, after the mentioned logarithmic rescaling, the
functional (\ref{uno}) 
reduces to 
\begin{alignat}{1}\label{eqfunctionaintrol}
\frac{1}{\ln (1/\e)}\left[ \int_{\Omega\times\Omega} \frac{1}{|x-y|^{n+1}}
\left| u(x)-u(y)\right|^2 dxdy
+\frac1\e\int_\Omega W(u(x)) dx \right]\,,
\end{alignat}
where $W:\R\to[0,\infty)$ is a multiwell potential
(and $\Omega\subset\R^n$). A problem of this kind 
was first studied by 
Alberti, Bouchitt\'e and Seppecher, for the case of a two-well potential 
\cite{AlbertiBouchitteSeppecher1994,AlbertiBouchitteSeppecher1998}. 
With this scaling, they proved a compactness result which shows that the
domain of the limiting functional is $BV(\Omega;\{W=0\})$. Further, they
proved $\Gamma$-convergence to a sharp-interface limit.
The crucial idea is that even though (\ref{eqfunctionaintrol}) is a nonlocal functional,
rearrangement can be used very efficiently. 
Indeed even though the problem is nonlocal they show by rearrangement that
optimal interface profiles are one-dimensional.
In particular, the leading-order
part of the energy arises from the nonlocal interaction of the area where $u$
is close to one minimum of $W$ with the area where $u$ is close to the other
one. The criticality of the singular kernel implies that all distances
contribute, and therefore that the overall interaction is logarithmic in the
distance of the two sets. For the same reason, the limiting energy does not
depend on the precise structure of the profile between the two sets. 

The case of infinitely many wells and anisotropic kernel was treated by two of us in 
\cite{GarroniMueller2006} (see also \cite{GarroniMueller2005}).
The compactness is more subtle due to the non-coerciveness of the multiwell
potential $\dist^2(u,\Z)$. The phenomenology is similar, and in particular
optimal interface profiles remain one-dimensional, and anisotropy gives
rise to an anisotropic line-tension energy of the form 
\begin{equation}\label{eqgammalimitscalarcase}
\int_{J_u} \gamma(\nu) |u^+-u^-| \, d\calH^{1}\,,\hskip1cm
u\in BV(\Omega;\Z)
\end{equation}
(in two spatial dimensions). Moreover, the line-tension energy density
$\gamma$ can be completely characterized in terms of the kernel $\Gamma$, i.e.
$$
\gamma(\nu)= 2\int_{\{x\cdot \nu=1\}} \Gamma(x)\, d\calH^1(x).
$$

In the present case the earlier rearrangement arguments do not apply, since
the phase field 
is vector-valued, and the nonlocal interaction is anisotropic
(note, however, that for certain vector-valued problems rearrangement
arguments can be used \cite{Alikakospc}).
Nonetheless one can abstractly prove that a $\Gamma$-limit exists, and that it
has the form 
\begin{equation}
\int_{J_u} \gamma(\nu,u^+-u^-) \, d\calH^{1}\,,\hskip1cm
u\in BV(\Omega;\Z^2)\,,
\end{equation}
but one does not have any further information on the line-tension energy
density $\gamma$ \cite{CacaceGarroni,Cacacethesis}. One can 
naively try to use the natural generalization of the formula derived in the
scalar case (\ref{eqgammalimitscalarcase}), namely,
\begin{equation}\label{eqdefgamma0intro}
\gamma_0(\nu,s)=2\int_{\{x\cdot\nu=1\}} s^T\Gamma(x)s\, d\calH^1(x)\,.
\end{equation}
However, this does, in general, not produce a lower semicontinuous functional
\cite{CacaceGarroni,Cacacethesis}, whereas the $\Gamma$-limit must be lower semicontinuous. This in 
particular implies that interfaces are more complicated and produce microstructure. The natural question
is whether the $\Gamma$-limit is characterized by the $BV$-relaxation of the 1D
interfacial 
energy (\ref{eqdefgamma0intro}) (see Remark 
\ref{RemarkdefBVrelaxation} below for details).

In this paper we assume that
\begin{equation}
\label{eq:derGammaintro}
\Gamma(z)=\frac{1}{|z|^{3}} \hat\Gamma\left(\frac{z}{|z|}\right)\,,
\end{equation}
where $\hat\Gamma\in L^\infty(S^1;\R^{N\times N}_+)$ obeys, for some $c>0$,
\begin{equation}\label{eqgammaposdeftheointro}
\frac{1}{c}|\xi|^2\le \xi\cdot \hat\Gamma(z)\xi\le c|\xi|^2 \hskip5mm 
\text{ for all } \xi\in \R^N, \, z\in S^1\,,
\end{equation}
and $\R^{N\times N}_+$
denote the positive definite, symmetric, $N\times N$ matrices. 
\begin{theorem}\label{theorem1}
Let $\Omega\subset\R^2$ be a bounded Lipschitz domain,
and suppose that the kernel $\Gamma$ satisfies (\ref{eq:derGammaintro}) and 
(\ref{eqgammaposdeftheointro}). Then
\begin{equation*}
\Gamma\hbox{-}\lim_{\e\to 0} 
\frac{1}{\ln (1/\e)} 
E_{\e}^* = E_0^\rel \,,
\end{equation*}
where
\begin{equation*}
E_0^\rel[u_0,\Omega]=
\begin{cases}
\displaystyle \int_{J_{u_0}\cap\Omega} \gamma_0^\rel(\nu,[u_0]) d\calH^1 &
\text{ if } 
u_0\in BV(\Omega;\Z^N)\\
\infty & \text{ else.}
\end{cases}
\end{equation*}
The surface energy $\gamma_0^\rel$ is the $BV$-relaxation of $\gamma_0$, as
defined in (\ref{eqdefgamma0intro}) (see 
Remark \ref{RemarkdefBVrelaxation} 
below for the definition of the $BV$-relaxation).
\end{theorem}
The corresponding compactness statement, namely, that sequences $u_\eps$ such
that $E_\eps^*[u_\eps,\Omega]/\ln(1/\eps)$ is bounded have a subsequence converging
to a limit $u\in BV(\Omega;\Z^N)$, can be immediately derived from the scalar
results in \cite{GarroniMueller2005,GarroniMueller2006} or from Proposition
\ref{propmakeubv} below.

Let us briefly sketch the strategy of our proof.
The difficulty is the proof of the lower bound. Due to the logarithmic
behaviour, the problem does not have an intrinsic natural scale, and so the
lower bound cannot be reduced to the study of an asymptotic cell problem
formula. The new idea in dealing with this kind of singular kernels is to
perform a dyadic decomposition of the kernel with a sequence of truncated
kernels. Each term in this decomposition is then regular, 
and one could hope to use the ideas of Alberti-Bellettini for non local phase
transition models with regular kernels
\cite{AlbertiBellettini1998a,AlbertiBellettini1998b}. But there 
is another obstacle in order to implement this strategy. In principle each
regular term in the decomposition could be optimized by very different
structures and the choice of one of them could produce a gross
underestimation. It is natural to conjecture that this 
does not happen, but this is not so easy to prove directly, and might
depend on finer details. 

We thus look for a more robust method which does not require such a detailed
analysis of the optimal structures. Roughly speaking we exploit the fact that
a $BV$ function cannot have significant microstructure on all scales
simultaneously. Since all scales participate roughly equally in the total
energy the few potentially bad scales can be ignored in the limit (see Section
\ref{outline} for a more detailed description of this idea). We focus
here on dimension two in view of the physical model which motivated 
our work. The decomposition strategy, however, is not restricted to two
dimensions. A related logarithmic decomposition strategy has also proved
useful in the codimension-two context of vortices in Ginzburg-Landau models,
see \cite{BethuelBrezisHelein1994,Struwe1994,SandierSerfaty2007} and references therein.

\subsection{Connection with a phase field model for dislocations}
\label{dislocations}
Functionals of the type under consideration arise in the study of phase field
models for dislocations inspired by the Peierls-Nabarro model (see
e.g. \cite{KCO}). 
Dislocations are line defects in crystals that are responsible for plastic
behaviour. They usually arise on special planes (the slip planes) that are
determined by the crystalline structure, and can be seen as the
discontinuities of a slip on this plane. Depending on the crystalline
structure on each slip plane several slip directions (Burgers vectors) are
possible, so that the slip can be represented as a vector-valued function
whose components represent the slip along a given Burgers vector. The idea of
the Peierls-Nabarro model, originally formulated for a one dimensional
problem, is to express the free energy in terms of the slip $u$ as follows 
$$ 
E_{\rm free}[u]= E_{\rm elastic}[u] + E_{\rm interfacial}[u]\,,
$$
where the first term represents the long-range elastic distortion due to the
slip and the second term is a nonlinear interfacial potential that remembers
the crystal lattice and penalizes slips that are not integer multiples of the
Burgers vectors. The main interest of this model is the persistence of
discrete features in a continuum setting. 
The reformulation of this model proposed by Koslowski-Cuiti{\~n}o-Ortiz
\cite{KCO,KoslowskiOrtiz2002} for the case of dislocations on a given slip
plane, using $N$ different slip systems determined by the Burgers vectors
${b}_1$, ..., ${b}_N$, considers slips 
$$
u_1{b}_1+...+u_N{b}_N
$$
where $u:\Omega\subset\R^2\to \R^N$. 
The interfacial energy favours values of $u$ close to $\Z^N$.
The nonlocal part is the bulk
elastic energy, which is given by the integral over the three-dimensional set
$\Omega\times\R$ of a quadratic form of the 
gradient of the displacement $U:\Omega\times\R\to\R^3$
induced by the slip $u$. 
Precisely, $U$ minimizes
the elastic energy $\int_{\Omega\times \R} \langle C\nabla U, \nabla U\rangle\, dx$ over all
vector fields which jump by $\sum u_ib_i$ on 
$\Omega\times\{x_3=0\}$. In the case of isotropic materials this reduces to
$$ \int_{\Omega\times \R} \frac{\mu}{2} |e(U)|^2+\lambda |{\rm tr}(e(U)|^2dx\,,
$$
where $e(U)=\frac{\nabla U+\nabla 
U^t}{2}$ is the linearized strain. 
Minimizing out $U$ leads to a nonlocal functional of $u$ of the kind of 
(\ref{uno}) (with some differences due to boundary effects, which do
not influence the leading-order behavior, see 
\cite{GarroniMueller2005,GarroniMueller2006}).

The connection with this application on dislocations produces very
interesting examples for the functional (\ref{uno}). In particular the kernel
arising from isotropic elastic interaction can be explicitly computed
(up to the boundary terms) for
different sets of Burgers vectors. For instance in the
case of a 
pair of orthogonal Burgers vectors (corresponding to square
symmetry) the explicit computation shows that the matrix-valued kernel
defined in (\ref{eqdefgamma0intro}) takes
the form \cite{CacaceGarroni}
\begin{equation}\label{eqgammacubic}
\gamma(\nu,s) = \frac{1}{4\pi (1-\tilde\nu)}
s\cdot \begin{pmatrix}
2-2\tilde\nu \sin^2\theta & \tilde\nu \sin 2\theta\\
\tilde\nu \sin 2\theta & 2-2\tilde\nu \cos^2\theta
\end{pmatrix} s\,.
\end{equation}
In this equation $\tilde\nu\in[-1,1/2]$ is the materials' Poisson ratio, and
$\theta$ 
characterizes the direction of the normal $\nu=(\cos\theta,\sin\theta)$ to the
interface. 
Notice that the given quadratic form is 
positive definite but the
off-diagonal entries are, for some values of $\tilde\nu$ and $\theta$, nonzero.

Consider now for example an interface in direction
$\nu=(\cos\theta,\sin\theta)$ between a region 
where $u$ equals $u_0=(0,0)$ and one
where $u$ equals $u_1=(1,1)$. The energy per unit length is given by
\begin{equation}\label{u1u0gamma}
\gamma(\nu, u_1-u_0)=
(u_1-u_0)\cdot
\hat\gamma\,
(u_1-u_0)
=\hat\gamma_{11}+2\hat\gamma_{12}+\hat\gamma_{22}\,,
\end{equation}
where $\hat\gamma$ is the matrix representation of the  quadratic form
$\gamma(\nu, \cdot)$.
If a thin layer where $u$ takes the value $u_2=(0,1)$ is inserted in between
(see Figure \ref{fig:cacace}, middle panel), then the sum of the two
interfaces has the energy
\begin{equation}\label{u1u0u2gamma}
\gamma(\nu, u_2-u_0)=
(u_2-u_0)\cdot \hat\gamma\,
(u_2-u_0)
+
(u_1-u_2)\cdot
\hat\gamma \,(u_1-u_2)
=\hat\gamma_{11}+\hat\gamma_{22}\,.
\end{equation}
If $\gamma$ takes the form (\ref{eqgammacubic}), then
 one or the other is more
convenient depending on the sign of $\tilde\nu\sin2\theta$. It is therefore
clear that the relaxation will choose for each 
direction the optimal decomposition of the total jump.
Cacace and Garroni \cite{CacaceGarroni} have shown that for some interfaces a more complex
relaxation takes place, and in particular that in some directions a smaller
energy is achieved by inserting fine-scale oscillations in the interface (see
Figure \ref{fig:cacace}, right panel). The intermediate $(0,1)$ layer is then inserted only in the
part of the interface in which $\tilde\nu\sin2\theta$ is positive. 
Their
construction proves that the $BV$-relaxation of the surface energy obtained for
one-dimensional interfaces is nontrivial. It is possible to prove that this
oscillatory construction indeed gives the $BV$-relaxation for this case. 

\begin{figure}[t]
\centering
\psfrag{labeln}{ {\tiny$ 
\begin{pmatrix}
\cos\theta\\ \sin\theta
\end{pmatrix}
$}}
\psfrag{labelu0}{{ \tiny $ u\!\!=\!\!
\begin{pmatrix}
0\\0 
\end{pmatrix}
$}}
\psfrag{labelu1}{{ \tiny $ u\!\!=\!\!
\begin{pmatrix}
1\\1
\end{pmatrix}
$}}
\psfrag{labelu2}{{ \tiny $ u\!\!=\!\!
\begin{pmatrix}
0\\1
\end{pmatrix}
$}}
\includegraphics[width=0.3\textwidth]{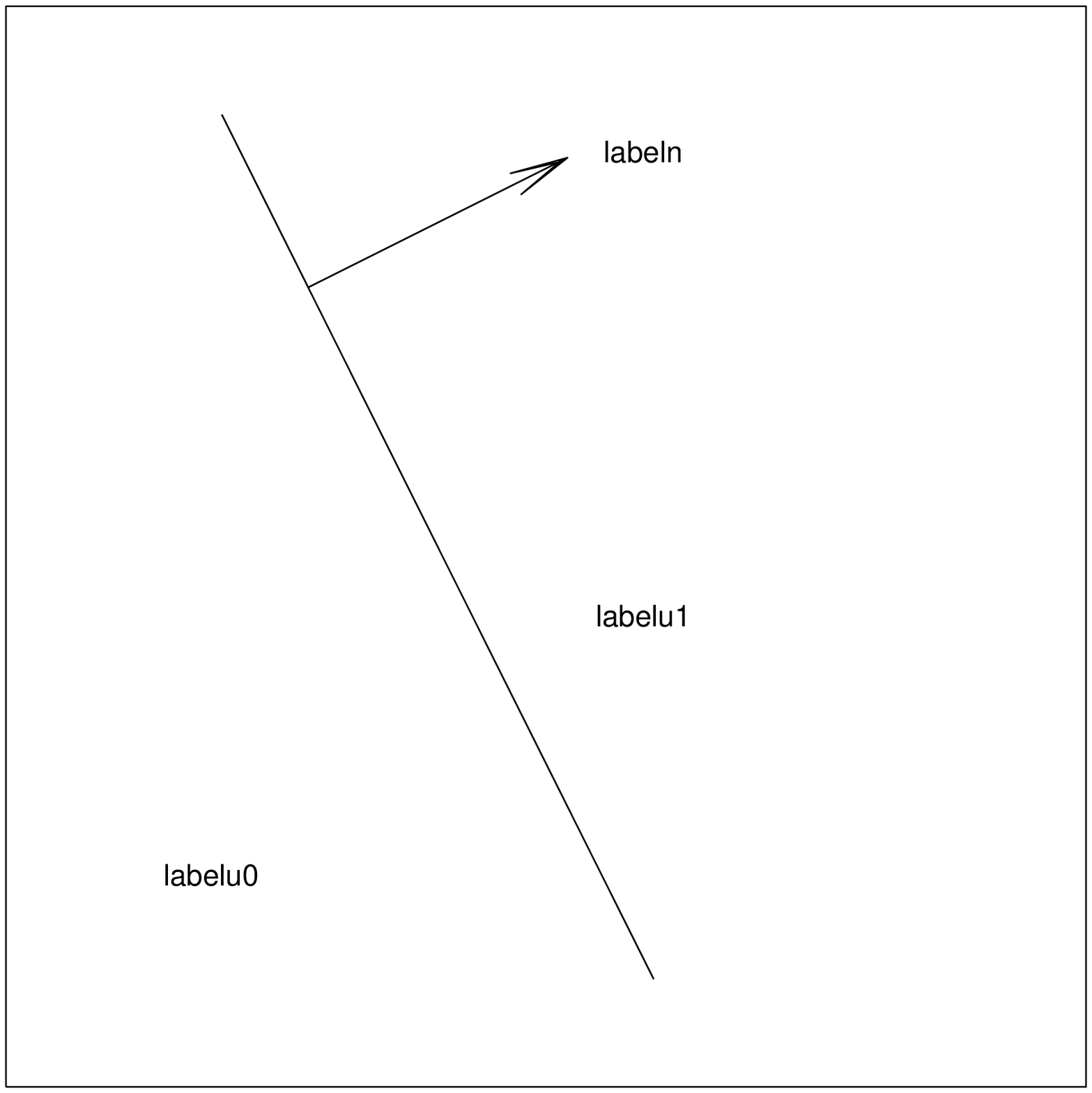}
\includegraphics[width=0.3\textwidth]{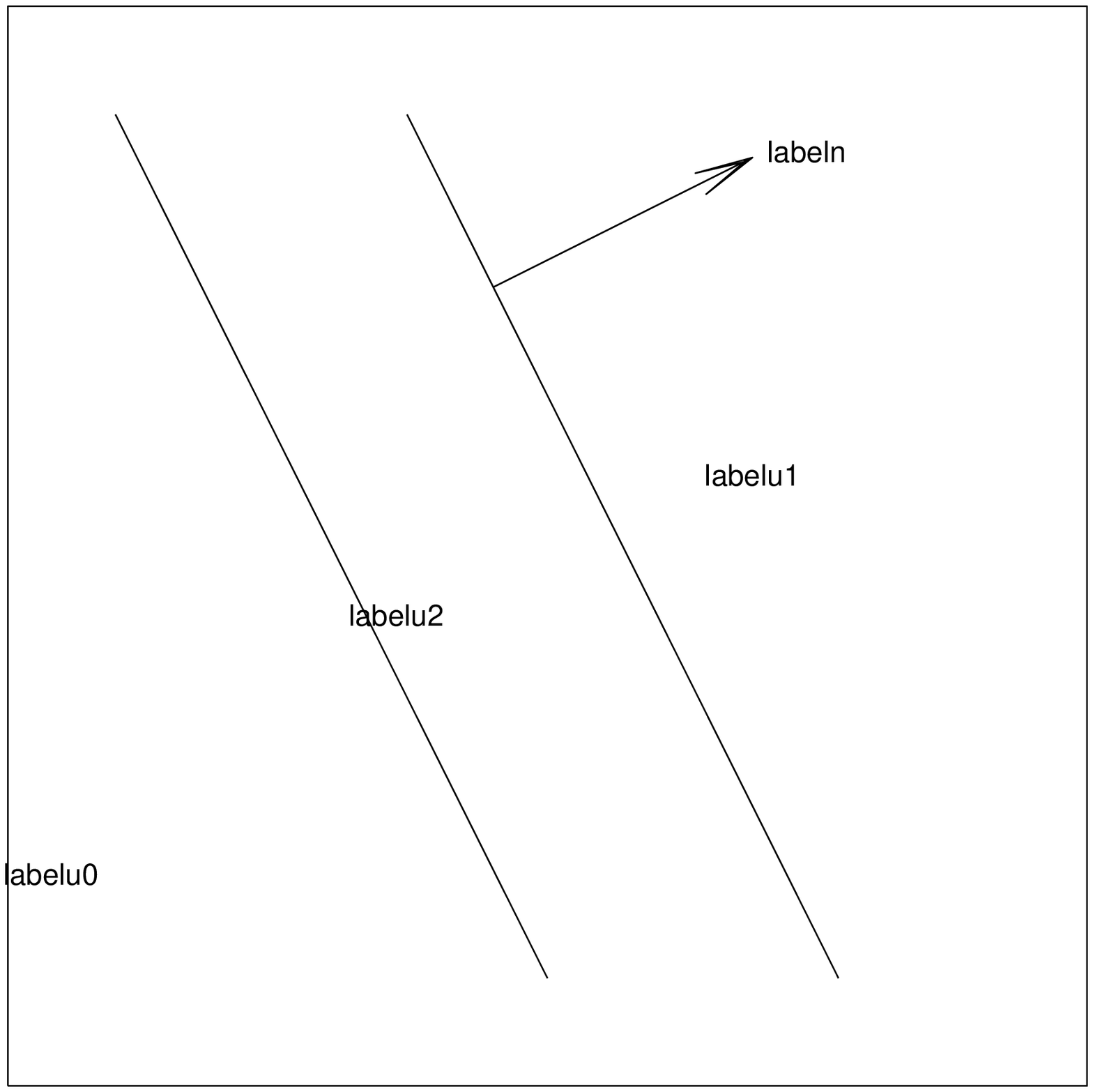}
\includegraphics[width=0.3\textwidth]{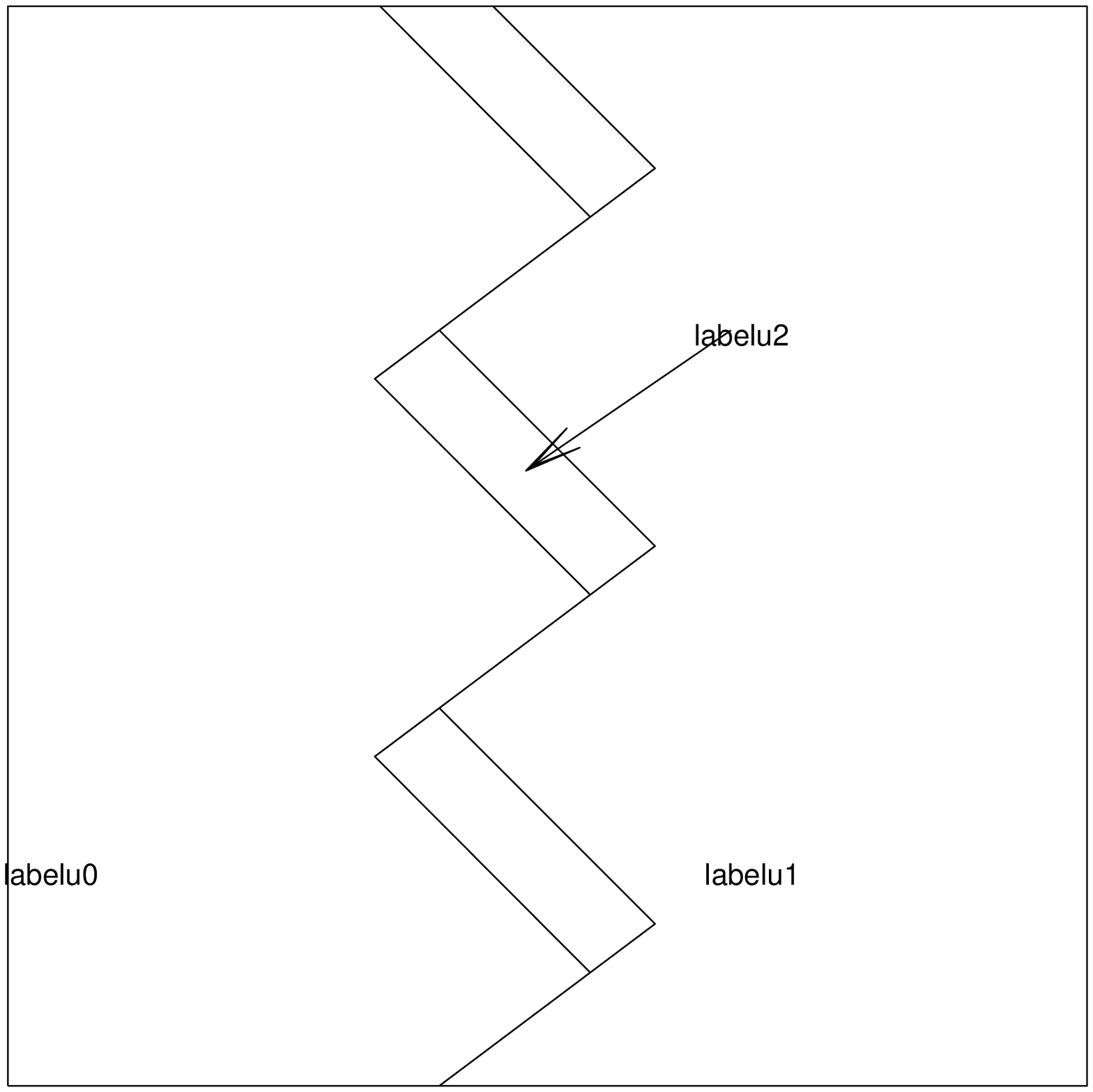}
\caption{Left: Sketch of the interface as in (\ref{u1u0gamma}). Middle:
interface in (\ref{u1u0u2gamma}). Right: oscillatory interfacial profile
corresponding to a macroscopic vertical interfaces which combines the two
options. 
\label{fig:cacace}}
\end{figure}

\section{Outline of the proof }\label{outline}

We consider, for $u:\Omega\subset\R^2\to\R^N$, 
\begin{alignat}{1}
\nonumber E_\e[u,\Omega]=\frac{1}{\ln1/\e}&
\left[ \int_{\Omega\times\Omega} \Gamma_{ij}(x-y)
\left[ u_i(x)-u_i(y)\right] \left[ u_j(x)-u_j(y)\right] dxdy\right.\\
& \left.
+\frac1\e\int_\Omega\dist^2(u(x), \Z^N) dx \right]
\end{alignat}
(sum over $i,j$ from 1 to $N$ is implicit). 
This differs from (\ref{uno}) in that the logarithmic factor is incorporated.

The upper bound follows directly from the abstract representation result
\cite{CacaceGarroni,Cacacethesis} and
the analysis of one-dimensional interfaces (see Section \ref{upperbound}).

The main point in the proof of Theorem~\ref{theorem1} is to establish the
following lower bound, whose proof will be concluded in Section
\ref{iterativemoll}. 

\begin{theorem}\label{theorem2}
Let $\Omega\subset\R^2$ be a bounded Lipschitz domain, and assume $u_0\in
BV(\Omega;\Z^N)$. Then for any sequences $\e_i\to0$, $u_i\to u_0$ in
$L^1(\Omega;\R^N)$ we have
\begin{equation*}
\liminf_{i\to\infty} E_{\e_i}[u_i,\Omega] \ge E_0^\rel [u_0,\Omega]\,,
\end{equation*}
where
\begin{equation*}
E_0^\rel[u_0,\Omega]=\int_{J_{u_0}\cap\Omega} \gamma_0^\rel(\nu,[u_0]) d\calH^1\,.
\end{equation*}
\end{theorem}

\begin{remark}\label{RemarkdefBVrelaxation}
The surface energy $\gamma_0^\rel$ is the $BV$-relaxation of $\gamma_0$ as
defined in (\ref{eqdefgamma0intro}) and is given by
\begin{alignat*}{1}\label{BV-rel}
\gamma_0^\rel(\nu,s)=\min\left\{\int_{\overline{Q}_\nu\cap J_v}\!\!\!\! \gamma_0(\nu_v,[v]) d\calH^1:\ v\in BV_{\rm loc}(\R^2;\Z^N) \ \hbox{and } v=u_\nu^s\ \hbox{in }Q^c_\nu\right\}
\end{alignat*}
where $Q_\nu$ is a unit square with two sides parallel to $\nu$ and $u_\nu^s=s\chi_{\{x\cdot \nu>0\}}$. 
The energy $E_0^\rel[u,\Omega]$ is then the lower-semicontinuous envelope of 
$$
\int_{J_u\cap\Omega} \gamma_0(\nu,[u]) d\calH^1
$$ 
with respect to the $L^1$ topology.
For more details about the relaxation of functionals defined on partitions we refer to \cite{AmbrosioBraides1990a,AmbrosioBraides1990b,AmbrosioFP}.
\end{remark}

The main ideas for the proof of Theorem~\ref{theorem2} are the following.

We first show that the kernel $\Gamma$ can be rewritten by a dyadic
superposition of truncated kernels and that given any function $u\in W^{1,1}$
and any given level of truncation, $u$ can be substituted by a $BV$ function
with values in $\Z^N$ whose truncated energy is controlled by the energy of
$u$. 
Then we show that one dimensional functions with values in $\Z^N$ are good
test functions for computing the truncated energy, which in turn can be
expressed in term of the line tension $\gamma_0$. 
In general functions with controlled energy do not satisfy the property of
being one dimensional, but we can show that this is almost true locally if
their total variation does not change much after mollification. We show this
last property for a sequence of mollifications on suitably well separated
scales of our initial sequence. The key idea is that a sequence with
controlled energy can oscillate at many scales, but not at all scales. 
To illustrate this strategy we first apply it to the one dimensional case in Section~\ref{onedimensional}.

In Section \ref{secpreliminary} we recall some elementary results for nonlocal
terms with integrable kernel. In Section \ref{layering} we decompose the
singular kernel into a sequence of integrable kernels and show that the
sequence $u_k$ in the $\liminf$ can be essentially replaced by a sequence $v_k$ with
values in $\Z^n$ which is uniformly bounded in $BV$. In Section \ref{1D} we
restrict ourselves to one-dimensional functions of the form $w_k(x)=f(x\cdot
\nu)$ and show that for those the limit energy can be computed explicitly. Our
general philosophy is that on most scales the given function $v_k$ is close
to a one-dimensional function. Thus in Section \ref{local1D} we carefully
estimate the energy of almost one-dimensional functions. In Sections
\ref{control} and \ref{iterativemoll} we combine those estimates with the idea
that on most length scales a $BV$ function is locally close to a one-dimensional
function. To quantify the distance of the $BV$ function from a locally
one-dimensional function on a given length scale we use an iterative
mollification on the different length scales, starting from the smallest one,
and measure the defect in the total variation of the gradient (see Section \ref{iterativemoll}
for the details).

\section{Elementary estimates on the nonlocal term}\label{secpreliminary}

\begin{lemma}\label{lemmaeasy}
Given $\Gamma'\in L^1(\R^2;\R^{N\times N}_+)$ and 
$u\in L^2(\Omega;\R^N)$ we define
\begin{equation}\label{eqdefp}
p_{\Gamma',\Omega}(u) = 
\sum_{i,j=1}^N \int_{\Omega\times\Omega} \Gamma_{ij}'(x-y)
\left[ u_i(x)-u_i(y)\right] \left[ u_j(x)-u_j(y)\right] dxdy\,.
\end{equation}
Then:
\begin{enumerate}
\item 
One has
\begin{equation}
0\le p_{\Gamma',\Omega}(u) \le 4
\|\Gamma'\|_{L^1(\R^2;\R^{N\times N})}\|u\|_{L^2(\Omega;\R^{N})}^2\,. 
\end{equation}
\item\label{lemmaeasy2}
The function $p^{1/2}_{\Gamma',\Omega}(\cdot)$ is a seminorm, and in
particular 
\begin{equation}
p_{\Gamma',\Omega}(u')\le (1+\eta) p_{\Gamma',\Omega}(u) +
(1+\frac{1}{\eta}) p_{\Gamma',\Omega}(u-u')
\end{equation}
for all $\eta>0$, $u$, $u'\in L^2(\Omega;\R^N)$. 
\item The function $p$ is set-superadditive, in the sense that for any pair $A$,
$B\subset 
\R^2$, with $A\cap B=\emptyset$, one has
\begin{equation}
p_{\Gamma',A}(u)+ p_{\Gamma',B}(u)\le p_{\Gamma',A\cup B}(u)\,.
\end{equation}
\end{enumerate}
\end{lemma}
In the following we write $\|\cdot \|_{L^p(\Omega)}$ or simply 
 $\|\cdot \|_{L^p}$ for $\|\cdot \|_{L^p(\Omega;\R^N)}$ or
$\|\cdot \|_{L^p(\Omega;\R^{N\times N})}$, when no ambiguity arises.
\begin{proof}
The upper bound follows from
\begin{equation*}
\| (\Gamma'\ast u) u\|_{L^1} \le \| \Gamma'\ast u\|_{L^2} \|u\|_{L^2} 
\le \|\Gamma'\|_{L^1} \|u\|_{L^2}^2\,.
\end{equation*}
The lower bound follows from the fact that $\Gamma'(z)$ is a positive
definite matrix. 

Since $p$ is a positive semidefinite, continuous quadratic form, its square
root is a seminorm.

Finally, observe that $(A\cup B)\times(A\cup B)= (A\times A) \cup (B\times B)
\cup (A\times B) \cup (B\times A)$, hence we only have to show that the
contributions of the last
two terms are nonnegative. Let $\xi^{xy}=u(x)-u(y)$. Then
\begin{alignat*}{1}
\sum_{ij}& \int_{A\times B} \Gamma_{ij}(x-y)
\left[ u_i(x)-u_i(y)\right] \left[ u_j(x)-u_j(y)\right] dxdy\\
&=
\sum_{ij} \int_{A\times B} \Gamma_{ij}(x-y)
\xi^{xy}_i \xi^{xy}_j \ge 0\,,
\end{alignat*}
since $\Gamma(x-y)\in \R^{N\times N}_+$ for any $x$, $y$. Since we may
exchange $A$ and $B$, this concludes the proof.
\end{proof}

\section{Dyadic decomposition and compactness}\label{layering}

In this section we show that we can represent the singular kernel with a
superposition of truncated kernels for which the regular phase field $u$ can be
substituted with a $BV$ function. 
%%%%% 

Let $\phi(x)=x^{-3}$, $\phi:(0,\infty)\to\R$. We consider the following dyadic decomposition
of $\phi$. For $k\in\N$ set
\begin{equation*}
\phi_k(x)=
\begin{cases}
2^{3(k+1)}-2^{3k} & \text{ if } 0<x \le 2^{-k-1}\\
x^{-3}-2^{3k} & \text { if } 2^{-k-1}<x \le 2^{-k}\\
0 & \text{ if } x>2^{-k}\,.
\end{cases}
\end{equation*}
Further, we set $\phi_{-1}(x)=1$ for $x<1$, and $\phi_{-1}(x)=x^{-3}$
otherwise. A 
simple check shows that $\phi=\sum_{k=-1}^\infty \phi_k$, and each
$\phi_k$ is continuous, nonnegative, and for $k\ge0$ the function $\phi_k$ is
supported on $\overline B_{2^{-k}}$. 
We denote by $\Gamma_k(z)=\phi_k(|z|)\hat\Gamma(z/|z|)$ the ``layer'' kernel,
and by $\Gamma_{0,k}=\sum_{i=0}^k \Gamma_i$ the truncated kernel (we
shall not need to use the function $\phi_{-1}$ explicitly). 
For later reference we remark that
\begin{equation}
\|\Gamma_k\|_{L^1(\R^2)} = c 2^k\,,\hskip1cm \text{ for all }k\in\N\,.
\end{equation}
We shall replace a function with good energy by a $BV$ function which takes
integer values and which has good truncated energy. The truncated energy is
defined by 
\begin{equation*}
E_\e^k[v,\Omega]=\frac{1}{\ln 1/\e} p_{\Gamma_{0,k},\Omega}(v)\,.
\end{equation*}

\begin{proposition}\label{propmakeubv}
Assume that $\hat\Gamma$ is strictly positive definite, i.e., that there is $c>0$
such that 
\begin{equation}\label{eqgammaposdef}
\xi\cdot \hat\Gamma(z)\xi\ge c|\xi|^2 \hskip5mm
\text{ for all } \xi\in \R^N, \, z\in S^1\,. 
\end{equation}
Let $\Omega\subset\R^2$ be a bounded Lipschitz domain, and
$\omega\subset\subset\Omega$, $\delta\in(0,1/2)$. 
Then for every  sufficiently small $\e>0$ (on
a scale set by $\delta$ and $\dist(\omega,\partial\Omega)$) and every
 $u\in L^2(\Omega;\R^N)$ 
there are $k\in\N$ and 
$v\in BV(\omega;\Z^N)$ such that
\begin{equation*}
E_\e^k[v,\omega]\le E_\e[u,\Omega]\left( 1 + \frac{C}{\delta (\ln
1/\e)^{1/2}}\right)\,, 
\end{equation*}
\begin{equation}\label{eqpropmakeubvbvnorm}
|Dv|(\omega)\le \frac{C}{\delta} E_\e[u,\Omega]\,,
\end{equation}
and 
\begin{equation*}
\e^{1-\delta/2}\le 2^{-k} \le \e^{1-\delta}\,.
\end{equation*}
Furthermore,
\begin{equation*}
\|u-v\|_{L^1(\omega;\R^N)} \le C 2^{-k/2} ( E_\e[u,\Omega])^{1/2}\,.
\end{equation*}
\end{proposition}
\begin{proof}
We consider, for $k\in\N$, the quantities
\begin{equation*}
p_{\Gamma_k, \Omega}(u)\,.
\end{equation*}
Clearly $E_\e[u,\Omega]\ge \frac{1}{\ln1/\e} \sum_{k=0}^\infty
p_{\Gamma_k, \Omega}(u)$. We can assume without loss of generality that
 $\eps$ is sufficiently small that the number of $k\in \N$
such that $\e^{1-\delta/2}\le 2^{-k} \le \e^{1-\delta}$
is at least $\delta \ln (1/\eps)/(4 \ln 2)$. 
Therefore there is one $k$ such that
\begin{equation}\label{eqchoicek}
\e^{1-\delta/2}\le 2^{-k} \le \e^{1-\delta} \text { and } p_{\Gamma_k, \Omega}(u) \le
\frac{C}{\delta}
E_\e[u,\Omega]\,. 
\end{equation}
Let $\alpha=2^{-k-4}$, and for $z\in \alpha \Z^2$ consider 
$q_z=z+(0,\alpha)^2$ and $Q_z=z+(-\alpha, 2\alpha)^2$.
Let $ Z=\{z\in\alpha\Z^2: Q_z\subset\Omega\}$. 
We observe that $\omega$ is covered by the disjoint union of the
small squares $\{q_z\}_{z\in Z}$ (up to a null set) and that the large
squares $Q_z$ have finite overlap and are contained in $\Omega$.

We claim that for any $z\in Z$ there is $v_z\in \Z^N$ such that
\begin{equation}\label{eqreplaceubyconstant}
\int_{Q_z} |u-v_z|^2 dx \le c \int_{Q_z} \dist^2( u,\Z^N)\, dx
+c 2^{-k} p_{\Gamma_k, Q_z}(u)\,,
\end{equation}
for some constant $c$ depending only on $\hat\Gamma$ and $N$.

Since $\diam(Q_z)< 2^{-k-1}$, for any $x,y\in Q_z$ we have
$\phi_k(x-y)= 2^{3(k+1)}-2^{3k}$. Recalling (\ref{eqgammaposdef}) we obtain
\begin{equation}\label{eqpgammakQz}
p_{\Gamma_k, Q_z}(u) \ge c 2^{3k} \int_{Q_z}\int_{Q_z} |u(x)-u(y)|^2 dx dy 
\ge c 2^k \int_{Q_z} |u(x)-\bar u|^2 dx
\end{equation}
where $\bar u$ is the average of $u$ over $Q_z$.
Fix $w:Q_z\to \Z^N$ measurable and such that $\dist(u,\Z^N)=|u-w|$, and let
$\bar w$ be its average. We estimate
\begin{alignat*}1
\int_{Q_z} |w-\bar w|^2 dx \le& 3
\int_{Q_z} |w- u|^2 dx +3\int_{Q_z} |u- \bar u|^2 dx +3 \int_{Q_z}|\bar u- \bar w|^2 dx \,.
\end{alignat*}
The last term is controlled by the first term in the right-hand side, which in
turn is controlled by 
the integral of the squared distance of $u$ from $\Z^N$. The second term in the right-hand side is controlled
by (\ref{eqpgammakQz}). Therefore
\begin{alignat*}{1}
\int_{Q_z} |w-\bar w|^2 dx \le& 6 \int_{Q_z} \dist^2(u,\Z^N) \, dx +
c 2^{-k} p_{\Gamma_k, Q_z}(u)\,.
\end{alignat*}
Recalling that $w\in \Z^N$, we get 
\begin{alignat*}1
\calL^2(Q_z)\dist^2(\bar w, \Z^N)\le 
\int_{Q_z} |w-\bar w|^2 dx \le& 
6 \int_{Q_z} \dist^2( u,\Z^N)
+c 2^{-k} p_{\Gamma_k, Q_z}(u)\,.
\end{alignat*}
We pick $v_z\in \Z^N$ such that
$|\bar w-v_z|= \dist(\bar w, \Z^N)$, and obtain
\begin{alignat*}1
\int_{Q_z} |u-v_z|^2 dx 
\le& 3
\int_{Q_z} |u-\bar u|^2 dx + 3
\int_{Q_z} |\bar u-\bar w|^2 dx +3
\int_{Q_z} |\bar w-v_z|^2 dx\,.
\end{alignat*}
Collecting the previous estimates proves (\ref{eqreplaceubyconstant}).

Repeating the same procedure for all squares
we obtain a function $v\in L^\infty(\omega;\Z^N)$, defined by $v=v_z$ on
$q_z$, such that 
\begin{equation}\label{equvl2}
\|u-v\|_{L^2(\omega)}^2 \le\sum_{z\in Z}\|u-v\|_{L^2(Q_z)}^2 \le c 2^{-k} p_{\Gamma_k, \Omega}(u)
+c\int_{\Omega} \dist^2(u,\Z^N)dx\,.
\end{equation}
Here we used the superadditivity of $p_{\Gamma_k,\Omega}$
and the fact that the $Q_z$ have finite overlap.

We now turn to the estimate of the measure $|Dv|$. This is obviously
concentrated on 
the union of the boundaries of the squares. 
Consider two neighbouring squares $q_z$ and $q_{z'}$
(so that they share an edge, i.e., $z\ne z'$ and $\calH^1(\partial
q_z\cap \partial q_{z'})>0$). 
Then $q_z$ is contained in both $Q_z$ and $Q_{z'}$, and analogously
$q_{z'}$. The key idea is that 
if $u$ is approximately constant on each of the larger cubes, then the jump
must be zero (approximately constant on one of the large cubes does not suffice, with the
present definition of $v_z$ -- consider, for example, $u=0$ on $Q_z$ and
$u=100$ on $\R^2\setminus Q_z$). 
Precisely, 
\begin{alignat*}1
\calL^2(q_z) |v_z-v_{z'}|^2 \le& 2 \int_{q_z} |u-v_z|^2 + |u-v_{z'}|^2 dx\\
\le& 2 \int_{Q_z} |u-v_z|^2 dx + 2\int_{Q_{z'}}|u-v_{z'}|^2 dx\\
\le & c \int_{Q_z \cup Q_{z'}} \dist^2( u,\Z^N)\, dx
+c 2^{-k} p_{\Gamma_k, Q_z\cup Q_{z'}}(u)\,,
\end{alignat*}
where in the last step we used (\ref{eqreplaceubyconstant}).

Recalling that $v$ is integer-valued we obtain, for the same squares,
\begin{alignat*}1
|Dv|(\partial q_z\cap \partial q_{z'})&=
2^{-k}|v_z-v_{z'}|
\le 2^k \calL^2(q_z) |v_z-v_{z'}|^2\\
&\le c 2^{k} \int_{Q_z \cup Q_{z'}} \dist^2( u,\Z^N)\, dx
+c p_{\Gamma_k, Q_z\cup Q_{z'}}(u)\,.
\end{alignat*}
Summing over all squares gives
\begin{equation*}
|Dv|(\omega) \le c p_{\Gamma_k, \Omega}(u) +
c 2^k \int_{\Omega} \dist^2(u,\Z^N)dx\,.
\end{equation*}

% and hence we have
% \begin{equation*}
% \text{ either }
% p_{\Gamma_k, Q_z}(u)\ge c2^{-k} \text{ or }
% \int_{Q_z}\dist^2(u,\Z^N) \ge c \calL^2(Q_z)=c 2^{-2k}\,.
% \end{equation*}
% Further, if the jump is large then the first term is necessarily large, in the
% sense that
% \begin{alignat*}{1}
% \max [v]^2|_{\partial q_z} \le C 2^{-k} p_{\Gamma_k, Q_z}(u)+1\,.
% \end{alignat*}
% %%
% %%I don't understand the estimate above, while the equation below is a direct consequence of (\ref{stima}), that I added to explain a little bit the alternative, but maybe this can be done more efficiently o maybe more detailed 
% %%
% Therefore we conclude that
% \begin{equation*}
% |Dv|(\partial q_z) \le c p_{\Gamma_k, Q_z}(u)
% +c 2^k \int_{Q_z}\dist^2(u,\Z^N) dx\,.
% \end{equation*}
% Summing over all cubes gives
% \begin{equation*}
% |Dv|(\omega) \le c p_{\Gamma_k, \Omega}(u) +
% c 2^k \int_{\Omega} \dist^2(u,\Z^N)dx\,.
% \end{equation*}

>From (\ref{eqchoicek}) we obtain, for sufficiently small $\e$, 
\begin{equation}
\label{eq:comparekepslogeps}
2^k\le \frac{1}{\e^{1-(\delta/2)}}\le \frac{1}{\e \ln 1/\e}\,,
\end{equation}
and  therefore
\begin{equation}\label{eqestimatefinalDv}
|Dv|(\omega) \le \frac{C}{\delta} E_\e[u]\,.
\end{equation}
This concludes the proof of (\ref{eqpropmakeubvbvnorm}).

We compute, recalling (\ref{equvl2}) and Lemma
\ref{lemmaeasy}, for all $\eta\in(0,1/2)$
\begin{alignat*}{1}
p_{\Gamma_{0,k},\omega}(v) &\le (1+\eta)
p_{\Gamma_{0,k},\omega}(u)+ \left(1+\frac1\eta\right) p_{\Gamma_{0,k},\omega}(u-v)\\
&\le (1+\eta)
p_{\Gamma_{0,k},\omega}(u)+ \frac2\eta \|\Gamma_{0,k}\|_{L^1(\R^2)}
\|u-v\|_{L^2(\omega)}^2\,.
\end{alignat*}
Since $ \|\Gamma_{0,k}\|_{L^1(\R^2)}\le c2^k$, using (\ref{equvl2}) and
arguing as done for (\ref{eqestimatefinalDv}) we obtain
\begin{alignat*}{1}
p_{\Gamma_{0,k},\omega}(v) 
&\le (1+\eta)
p_{\Gamma_{0,k},\omega}(u)+ \frac1\eta c 2^k 2^{-k} p_{\Gamma_k,
\Omega}(u)
+ \frac1\eta c 2^k \int_\Omega\dist^2(u,\Z^N)\, dx
\\
&\le (1+\eta) (\ln \frac1\eps)
 E_\e[u,\Omega]
+ \frac1\eta \frac{c}{\delta} E_\e[u,\Omega]\,. 
\end{alignat*}
Finally, 
\begin{equation*}
E_\e^k[v,\omega]=\frac{1}{\ln 1/\e} p_{\Gamma_{0,k},\omega}(v)
\le 
(1+\eta) E_\e[u,\Omega] + \frac{c}{\delta\eta\ln 1/\e} E_\e[u,\Omega]\,.
\end{equation*}
Taking $\eta=(\ln 1/\e)^{-1/2}$ gives
\begin{equation*}
E_\e^k[v,\omega] \le 
E_\e[u,\Omega] + \frac{c}{\delta(\ln 1/\e)^{1/2}} E_\e[u,\Omega]\,.
\end{equation*} 
Finally, from (\ref{equvl2}) we have
\begin{equation*}
\|u-v\|_{L^1(\omega)} \le c 2^{-k/2} \left[p_{\Gamma_k, 
\Omega}(u) +c2^k\int_{\Omega} \dist^2(u,\Z^N)dx\right]^{1/2}\,.%%%%%% ADRIANA 2^k 
\end{equation*}
Recalling (\ref{eq:comparekepslogeps}) we conclude
\begin{equation*}
\|u-v\|_{L^1(\omega)} \le c 2^{-k/2} \left(E_\e[u,\Omega]\right)^{1/2}\,.
\end{equation*}
%%It seems to me that there was a square root missing I hope this is true
\end{proof}

As consequence any given sequence converging in $L^1$ to a function in
$BV(\Omega;\Z^N)$ can be substituted with a sequence in $BV(\Omega;\Z^N)$ for
which we control the energy. More precisely we have the following
proposition. 

\begin{proposition}\label{eqseqBV}
Let $\Omega\subset\R^2$ be a bounded Lipschitz domain, and assume $u_0\in
BV(\Omega;\Z^N)$. Then for any $\delta\in(0,1/2)$, any sequences $\e_i\to0$,
$u_i\to u_0$ in $L^1(\Omega;\R^N)$ and any Lipschitz domain
$\omega\subset\subset\Omega$ there is a sequence $v_k\in BV(\omega;\Z^N)$
such that $v_k\to u_0$ in $L^1(\omega;\R^N)$,
\begin{equation*}
\liminf_{k\to\infty} \frac{1}{k} \sum_{h=0}^k
p_{\Gamma_h,\omega}(v_k) \le (1+2\delta) \ln 2 \liminf_{i\to\infty}
E_{\e_i}[u_i,\Omega] \,,
\end{equation*}
and
\begin{equation*}
|Dv_k|(\omega)\le C_\delta (\liminf_{i\to\infty}
E_{\e_i}[u_i,\Omega] +1) \,.
\end{equation*}
\end{proposition}
\begin{proof}
If the $\liminf$ equals $\infty$ there is nothing to prove.
By taking a subsequence we can assume that
\begin{equation*}
E_{\e_i}[u_i,\Omega] \le \lim_{i\to\infty} E_{\e_i}[u_i,\Omega] +1
\end{equation*}
for all $i$. We apply Proposition \ref{propmakeubv} to each $u_i$, and
obtain $k_{i,\delta}$ and $v_{i,\delta}$. The estimate on the total
variation is immediate.
From the condition on $k_{i,\delta}$ we obtain
\begin{equation*}
-k_{i,\delta}\ln 2 \le -(1-\delta) \ln \frac1{\e_i}
\end{equation*}
which implies
\begin{equation*}
\frac1{k_{i,\delta}} \le (1+2\delta) \frac{\ln 2}{\ln \frac1{\e_i}}\,.
\end{equation*}
Therefore 
\begin{equation*}
\frac{1}{k_{i,\delta}} p_{\Gamma_{0,k_{i,\delta}},\omega}(v_{i,\delta}) \le (1+2\delta)
\ln 2\, E_{\e_i}[u_i, \Omega] \left( 1 + \frac{C_\delta }{(\ln
1/\e_i)^{1/2}}\right)\,,
\end{equation*}
which gives
\begin{equation*}
\liminf_{i\to\infty} \frac{1}{k_{i,\delta}}
p_{\Gamma_{0,k_{i,\delta}},\omega}(v_{i,\delta}) \le (1+2\delta) \ln 2
\lim_{i\to\infty} E_{\e_i}[u_i, \Omega]\,.
\end{equation*}
Taking a further subsequence we can assume the map $i\mapsto k_{i,\delta}$ to
be nondecreasing. We set for every $K\in\N$
\begin{equation}
w_K = v_{j,\delta} \text{ where } j=\min\{i\in\N: k_{i,\delta}\ge K\}\,.
\end{equation}
Then $w_K\to u$ in $L^1$, $|Dw_K|(\omega)$ still obeys the desired bound, and
from the fact that $ \frac{1}{k_{i,\delta}} 
p_{\Gamma_{0,k_{i,\delta}},\omega}(v_{i,\delta})$ is a subsequence of $\frac{1}{K} 
p_{\Gamma_0,\omega}(w_K)$ we get
\begin{equation*}
\liminf_{K\to\infty} \frac{1}{K} \sum_{h=0}^K
p_{\Gamma_h,\omega}(w_K) \le
\liminf_{i\to\infty} \frac{1}{k_{i,\delta}}
p_{\Gamma_{0,k_{i,\delta}},\omega}(v_{i,\delta})
\end{equation*}
and hence the thesis follows.
\end{proof}

\subsection{Digression: the one-dimensional case without
  rearrangement} \label{onedimensional} 
We pause for a moment to illustrate how Proposition \ref{eqseqBV} can be used
to obtain the lower bound without the use of rearrangement in the
one-dimensional scalar case, i.e., for the functional
(\ref{eqfunctionaintrol}) with $n=1$, $\Omega=(-L,L)$. For simplicity we only
consider the two-well problem, i.e., we take
\begin{equation*}
F_\eps[u]=
\frac{1}{\ln(1/\eps)}\left[
 \int_{(-L,L)^2} \frac{ (u(x)-u(y))^2}{|x-y|^2}
dxdy + \frac{1}{\eps}\int_{-L}^L \dist^2(u(x),\{0,1\})dx \right]\,.
\end{equation*}
In this case the $\Gamma$-limit is $2\# \text{jumps}[u]$ for $u\in
BV(\Omega,\{0,1\})$, and $\infty$ otherwise. The upper bound is in this
situation immediate (it suffices to smooth the jumps on the scale $\eps$).

Assume $\eps_i\to0$, and $u_i$ to be a sequence such that
\begin{equation*}
E^*=\liminf_{i\to\infty}  F_{\eps_i}[u_i]
\end{equation*}
is finite. We may assume
that the functions $u_i$ take values in $[0,1]$, since projection of the
values to $[0,1]$ reduces the energy $F_\eps$.  Fix $\delta\in(0,1/2)$.
The construction of
Proposition 
\ref{eqseqBV} yields a sequence of characteristic functions $v_k$ such that
\begin{equation}
\label{eq:42neua}
\# \text{jumps} (v_k) \le C_\delta (E^*+1)\le C_\delta'
\end{equation}
and
\begin{equation}
\label{eq:42neub}
\liminf_{k\to\infty} \frac{1}{k} \sum_{h=0}^k p_{\Gamma_h,\omega} (v_k) \le
(1+2\delta) (\ln 2) E^* \,,
\end{equation}
where $\omega=(-L',L')$, with $0<L'<L$, 
\begin{equation*}
\Gamma_k(x)=
\begin{cases}
2^{2(k+1)}-2^{2k} & \text{ if } 0<|x| \le 2^{-k-1}\\
|x|^{-2}-2^{2k} & \text { if } 2^{-k-1}<|x|\le 2^{-k}\\
0 & \text{ if } |x|>2^{-k}\,,
\end{cases}
\end{equation*}
and $p$ is defined as in (\ref{eqdefp}). In particular (\ref{eq:42neua})
implies that the limit 
 $u_0$ is a characteristic function with finitely many jumps, and
one sees easily that it suffices to prove the lower bound for the case that
$u_0$ has a single jump, i.e., $u_0=\chi_{(0,L)}$. 

Suppose that $v_k$ has a jump at $\overline x$ and no other jump in
$I_h=(\overline x - 2^{-h}, \overline x + 2^{-h})\subset (-L',L')$, for some
$h\in\N$. A change of variables and an explicit integration give
\begin{alignat}1
p_{\Gamma_h,\omega}(v_k) &\ge 
p_{\Gamma_h,I_h}(v_k) \nonumber \\
& =2\int_{I_h^-}\int_{I_h^+} \Gamma_h(x-y) dx dy \nonumber \\
& =2\int_{-1}^0\int_0^1 \Gamma_0(x-y) dx dy=2\ln 2\,.
\label{eq:42neucc}
\end{alignat}
Thus, {\em if} at every scale $2^{-h}$ the function $v_k$ has a jump which is
isolated in the above sense we immediately conclude from (\ref{eq:42neub})
that
$E^*\ge 2/(1+2\delta)$ which gives the desired conclusion, since
$\delta>0$ was 
arbitrary. Now we cannot expüect that $v_k$ has an isolated jump at every
scale, but we will see that this is true at most scales, after a small
modification of $v_k$. 

To make this precise we use that the jump set $J=J_{v_k}$ contains only
finitely many points, with a bound independent of $v_k$. We now iteratively
cluster and remove points in $J$ as follows:
\begin{enumerate}
\item Set $J^{(k+1)} = J$ and $w^{(k+1)}=v_k$.
\item Given $J^{(h+1)}$ and $w^{(h+1)}$
define $J^{(h)}$ and $w^{(h)}$ as follows. An $\ell$-cluster is a maximal
sequence of points $x_1<x_2<\ldots <x_\ell$ in $J^{(h+1)}$ with
$x_{i+1}-x_i<2^{-h}$. Now we obtain $J^{(h)}$ by replacing each cluster with
odd $\ell$ by the leftmost point $x_1$ and each cluster with even $\ell$ by the
empty set. If $J^{(h+1)}=J^{(h)}$, set 
$w^{(h)}=w^{(h+1)}$. If $J^{(h+1)}\ne J^{(h)}$, let 
$w^{(h)}$ be the characteristic function which jumps at the points in
$J^{(h)}$ and agrees with $w^{(h+1)}$ outside the intervals $[x_1, x_\ell]$
defined by the $\ell$-clusters in $J^{(h+1)}$.
Thus
\begin{equation}
\| w^{(h)} - w^{(h+1)}\|_{L^1} =
\| w^{(h)} - w^{(h+1)}\|^2_{L^2} \le 
(\# J^{h+1}) 2^{-h} \le (\# J) 2^{-h}\,.
\end{equation}
\end{enumerate}
We say that a level $h$ is {\em critical} if $J^{(h)}\ne J^{(h+1)}$. Since
$J$ is finite we have
\begin{equation*}
\# \{h: h \text{ is critical}\} \le \# J\,.
\end{equation*}
If $h$ is not critical then all jumps of $w^{(h)}=w^{(h+1)}$ are $h$-isolated,
i.e., there is no other jump in a neighbourhood of size $2^{-h}$. Thus by 
(\ref{eq:42neucc}) 
\begin{equation}
\label{eq:42neuc}
p_{\Gamma_h, (-L', L')}(w^{(h)}) \ge 2\ln 2\,,
\end{equation}
if $h$ is not critical. 

We now would like to exploit that $w^{(h)}$ is very close to $v_k$ if $h$,
$h+1$, \dots, $h+m-1$ are not critical. Fix $m\in\N$. We say that $h$ is good
if $h$, $h+1$, \dots, $h+m-1$ are not critical. Thus
\begin{equation}
\label{eq:42neud}
\# \{h\in\{1, \dots, k\}: h \text{ good}\} \ge k - m \# J\,.
\end{equation}
At the same time, if $h$ is good $w^{(h)}=w^{(h+m)}$, and therefore
\begin{equation*}
\|w^{(h)} - v_k \|_{L^1} =
\|w^{(h+m)} - v_k \|_{L^1} \le 2 (\# J) 2^{-(h+m)}
\end{equation*}
(and the same for the squared $L^2$ norm, since we are dealing with
characteristic functions).
We compute, using Lemma \ref{lemmaeasy},
\begin{alignat}1
p_{\Gamma_h, \omega}(w^{(h)}) &= 
p_{\Gamma_h, \omega}(v_k+w^{(h)}-v_k) \nonumber \\
&\le 
(1+\eta) p_{\Gamma_h, \omega}(v_k)
+ (1+\frac1\eta) p_{\Gamma_h, \omega}(w^{(h)}-v_k)\nonumber \\
&\le 
(1+\eta) p_{\Gamma_h, \omega}(v_k)
+ (1+\frac1\eta) 2^h \|w^{(h)}-v_k\|_{L^2}^2\nonumber \\
&\le 
(1+\eta) p_{\Gamma_h, \omega}(v_k)
+2 (1+\frac1\eta) 2^{-m} \# J\,.
\label{eq:42neue}
\end{alignat}
Recalling (\ref{eq:42neuc}), (\ref{eq:42neud}), and (\ref{eq:42neue}) we obtain
\begin{alignat*}1
\left(1-\frac{m \# J}{k}\right) 2\ln 2 &\le \frac1k \sum_{k \text{ good}}
p_{\Gamma_h, \omega} ( w^{(h)}) \\
&\le 
\frac{1+\eta}{k} \sum_{h=1}^k p_{\Gamma_h, \omega} (v_k) + 2 \left(1 +
\frac1\eta\right) 2^{-m} \# J\,.
\end{alignat*}
Taking the limit $k\to\infty$ and recalling (\ref{eq:42neua}) and
(\ref{eq:42neub}) we get 
\begin{equation}
2\ln 2 \le (1+\eta) (1+2\delta) (\ln 2) E^* + 2 \left(1+\frac1\eta\right) 
2^{-m} C_\delta (E^*+1)\,.
\end{equation}
Since $m$, $\delta$, $\eta$ were arbitrary it follows that $2\le E^*$ as
desired.

In concluding this digression, we summarize the main points of the argument:
\begin{enumerate}
\item For most levels $h$ the function $v_k$ can be approximated by a function
  $w^{(h)}$ which is monotone on scale $2^{-h}$ (near the jump set).
\item The function $w^{(h)}$ is close to $v_k$ in $L^1$ with a bound that
  scales slightly   better than $2^{-h}$. 
\item The (truncated) energy of $w^{(h)}$ is controlled by the energy of $v_k$.
\end{enumerate}
We will use a similar argument in higher dimension. In that case the
approximations $w^{(h)}$ will be one-dimensional and monotone. The good levels
$g$ are selected by the condition that the local $BV$ norm does not change
under successive 
mollification on the scales $2^{-h-m}$, \dots, $2^{-h}$. 

\section{One-dimensional test functions}\label{1D}
In this section we show that if a function is one dimensional and takes values
in $\Z^N$ it is possible to estimate its nonlocal truncated energy with the
right line tension energy. Our efforts then will be devoted to show that these
properties are almost satisfied locally by any sequence of finite energy. 
%%Maybe too vague and not here

Given a scalar kernel $\Gamma'\in L^1(\R^2;\R)$ and an orientation $\nu\in
S^1$ we define the one-dimensional interfacial energy (per unit length) by
\begin{alignat*}{1}
\gamma_{1D}^{\Gamma'}(\nu) =& 2 \int_{\{x\cdot\nu\le 0\le y\cdot\nu\,,\,\,
x\wedge \nu=0\}} \Gamma'(x-y) d\calH^1(x) dy
\\
=& 2 \int_{[0,\infty)^2\times \R} \Gamma'( (t_1-t_2)\nu + s \nu^\perp) dt_1dt_2ds\,.
\end{alignat*}
\begin{lemma} \label{lemmagammaexplicit}
For $a\in \R^N$, $k\in \N$, and $\Gamma$ as in Section \ref{layering}, one has
\begin{equation*}
\gamma_{1D}^{a\cdot \Gamma_k a}(\nu) =2(\ln 2) \int_{\{x\cdot\nu =1\}} a\cdot
\Gamma(x) a \,d\calH^1(x)
=2(\ln 2)\int_{-\pi/2}^{\pi/2}
a\cdot\hat\Gamma(e_\theta)a\, \cos\theta\,  d\theta 
\,.
\end{equation*}
\end{lemma}
\begin{proof}
We consider  polar coordinates centered at $x=-t_1\nu$, and set
 $y=x+ \rho e_\theta$, $e_\theta=\cos\theta\nu+\sin\theta\nu^\perp$. Then
 $y\cdot\nu=-t_1+\rho \cos\theta$, and
\begin{alignat*}{1}
\gamma_{1D}^{a\cdot \Gamma_k a}(\nu) =&
2\int_0^\infty\int_0^{\infty}\!\! \rho
\phi_k(\rho)\int_{\{\rho \cos\theta\ge
  t_1\}}a\cdot\hat\Gamma(e_\theta)a\,d\theta\,d\rho \, dt_1\\  
=&2\int_0^{\infty}\!\! \rho^2 \phi_k(\rho)\,d\rho
\int_{-\pi/2}^{\pi/2}
a\cdot\hat\Gamma(e_\theta)a\, \cos\theta\,  d\theta 
\,.
\end{alignat*}
By a direct computation one sees that
$$
\int_0^{\infty} \rho^2 \phi_k(\rho)\,d\rho=\ln 2\,.
$$
This proves the second expression.

At the same time the set $\{x\cdot\nu=1\}$ can be parametrized by
$x=e_\theta/\cos\theta$, $\theta\in(-\pi/2,\pi/2)$. Therefore
\begin{equation*}
  \int_{\{x\cdot\nu =1\}} a\cdot
\Gamma(x) a \,d\calH^1(x) =\int_{-\pi/2}^{\pi/2} (\cos\theta)^3
a\cdot\hat\Gamma(e_\theta)a \frac{1}{\cos^2\theta} d\theta\,.
\end{equation*}
Collecting the previous expressions the proof is concluded.
\end{proof}

\begin{lemma}\label{lemma1Dinterfacesseparate}
Assume $u(x)=u_0+v_0\lambda(x\cdot\nu)$, $u\in BV(\R^2;\Z^N)$, with
$\lambda$ monotone and $\nu\in S^1$, $u_0, v_0\in\R^N$. For $m,k\in\N$ with
$k\ge m+1$,  set  $Q=[0,2^{-m}]^2$. Then 
\begin{equation}
p_{\Gamma_k,Q}(u) \ge \int_{q\cap J_u} |[u]|^2\frac{1}{|v_0|^2}
\gamma_{1D}^{v_0\cdot \Gamma_kv_0}(\nu) d\calH^1\,. 
\end{equation}
Here $q=[2^{-k}, 2^{-m}-2^{-k}]^2$.
\end{lemma}
\begin{proof}
We first write out
\begin{equation*}
p_{\Gamma_k,Q}(u) = \int_{Q\times Q} (v_0\cdot \Gamma_k(x-y) v_0) 
(\lambda(x\cdot \nu)-\lambda(y\cdot \nu))^2 dx dy\,.
\end{equation*}
Notice that by assumption $v_0\cdot\Gamma_k v_0\ge 0$ pointwise. 
By symmetry we can restrict to $x\cdot\nu\le y\cdot \nu$, and add a
factor 2. The last factor can be estimated, using the monotonicity of
$\lambda$, for $x,y\not\in J_u$, by 
\begin{equation*}
(\lambda(x\cdot\nu)-\lambda(y\cdot\nu))^2 = \left( \sum_{t\in J_\lambda\cap [
x\cdot\nu,y\cdot\nu]} [\lambda](t)\right)^2 \ge 
\sum_{t\in J_\lambda\cap [x\cdot\nu,y\cdot\nu]} [\lambda]^2(t)
\end{equation*}
(here $[a,b]$ is the  segment joining $a$ and $b$).
Therefore
\begin{equation*}
p_{\Gamma_k,Q}(u) \ge 2 \int_{(x,y)\in Q\times Q, (y-x)\cdot\nu\ge 0} 
\sum_{t\in J_\lambda\cap [x\cdot\nu,y\cdot\nu]} [\lambda]^2(t)
(v_0\cdot \Gamma_k(x-y) v_0) dxdy\,.
\end{equation*}
Swapping the sum with the integral, we obtain
\begin{equation*}
p_{\Gamma_k,Q}(u) \ge 
2 \sum_{t\in J_\lambda} [\lambda]^2(t)
\int_{(x,y)\in Q\times Q, x\cdot\nu \le t\le y\cdot\nu} 
(v_0\cdot \Gamma_k(x-y) v_0) dxdy\,.
\end{equation*}
At this point we have separated the different interfaces, and we can
deal with a single one. To conclude the proof it suffices to show that for any
$t\in J_\lambda$,  
\begin{equation}\label{eqIt}
2 \int_{(x,y)\in Q\times Q,  x\cdot\nu \le t\le y\cdot\nu} 
(v_0\cdot \Gamma_k(x-y) v_0) dxdy \ge \calH^1(I_t) \gamma_{1D}^{v_0\cdot \Gamma_kv_0}(\nu) \,,
\end{equation}
where $I_t=q\cap \{z: z\cdot \nu=t\}$ is the ``reduced'' interface.
Recalling that $\supp \Gamma_k\subset B_{2^{-k}}$ and $B_{2^{-k}}(q)\subset
Q$, we obtain, for any $x\in q$, 
\begin{equation*}
\int_{y\in Q, y\cdot\nu\geq t} 
(v_0\cdot \Gamma_k(x-y) v_0) dy= \int_{y\in \R^2, y\cdot\nu\geq t}
(v_0\cdot \Gamma_k(x-y) v_0) dy\,. 
\end{equation*}
We restrict in (\ref{eqIt}) the $x$ integration to the set $S_t=
I_t+[-2^{-k},0]\nu=
\{ z + w: z\in I_t, w\in [-2^{-k},0]\nu\}\subset Q$, and decompose the integral
into the 
component parallel and orthogonal to $\nu$. We obtain
\begin{alignat*}{1}
&2 \int_{(x,y)\in S_t\times \R^2, y\cdot\nu\ge t} 
(v_0\cdot \Gamma_k(x-y) v_0) dxdy\\
= &2 \calH^1(I_t) \int_{(x,y)\in ([t-2^{-k},t]\nu\times\R^2),
y\cdot\nu\ge t} 
(v_0\cdot \Gamma_k(x-y) v_0) d\calH^1(x)dy\,.
\end{alignat*}
Since $\supp \Gamma_k\in B_{2^{-k}}$ the integral in $x$ can be
extended to $(-\infty,t)\nu$. This concludes the proof.
\end{proof}

The next lemma deals with the reduction of one-dimensional functions to
integer-valued one-dimensional functions. 
\begin{lemma}\label{lemmaZn}
Let $\Omega\subset\R^n$ be bounded and measurable, $M>0$. 
Let $u:\R^n\to\R^N$ be of the form
\begin{equation*}
u(x)=a \lambda(x\cdot\nu) + b\,,
\end{equation*}
for some $a,b\in \R^N$, $\lambda\in L^\infty(\R;\R)$,
$\nu\in S^{n-1}$. If 
\begin{equation*}
\|a\lambda \|_{L^\infty(\Omega;\R^N)}\le M 
\end{equation*}
then there are $a^*, b^*\in \Z^N$, $\lambda^*\in L^\infty(\R;\Z)$ such that
the function $u^*(x)=a^* \lambda^*(x\cdot\nu) + b^*$ obeys
\begin{equation*}
\| u-u^*\|_{L^1(\Omega;\R^N)}\le C \| \dist(u,\Z^N)\|_{L^1(\Omega)}\,.
\end{equation*}
Here $C$ depends only on $N$ and $M$.
\end{lemma}
Notice that the $L^\infty$ bound is needed, as the following example
on $\Omega=(0,3)$ shows:
\begin{equation*}
b=0\,,\hskip5mm 
a=\vectdue{1}{1/k}\,,\hskip5mm 
\lambda(x)= 
\begin{cases}
0 & \text{ if }x\in (0,1]\\
1 & \text{ if } x\in (1,2]\\
k & \text{ if } x\in (2,3)\,.
\end{cases}
\end{equation*}
Here $\|\dist(u_k, \Z^2)\|_{L^1}=1/k$, but for any $u^*_k$ as stated one
has $\|u_k-u^*_k\|_{L^1}\ge 1/2$. Indeed, since the three values
$(0,0)$, $(1,0)$, $(k,1)$ do not lie on a straight line, $u^*_k$ cannot
take all three of them; hence at least one entry must be off by at
least $1/2$.
\begin{proof}[Proof of Lemma~\ref{lemmaZn}]
Let
\begin{equation*}
\eta=\|\dist(u,\Z^N)\|_{L^1(\Omega)}\,.
\end{equation*}
We can assume without loss of generality that $|a|=1$ and $|b|\le N$
(otherwise we prove the lemma for the function $v(x)=u(x)-[b]$, where $[b]$
denotes a vector whose components are the integer parts of those of $b$). 
We define $z:\R\to \Z^N$ measurable and such that
$\dist(a \lambda(t) + b,\Z^N)=|a \lambda(t) + b-z(t)|$, for all
$t\in\R$. Clearly $\|z\|_{\infty}\le M+2N$. 

For $w\in \Z^N\cap B_{M+2N}$, define
\begin{equation*}
\Omega(w) = \{x\in \Omega: z(x\cdot\nu)=w\}\,,
\end{equation*}
so that
\begin{equation}\label{eqdefwlambdaeta}
\|\dist(u,\Z^N)\|_{L^1}= \sum_w \|a\lambda(x\cdot\nu)+b-w\|_{L^1(\Omega(w))} = \eta\,.
\end{equation}
Choose $w_1\ne w_2$  such that
\begin{equation*}
\calL^n(\Omega(w_1))\ge \calL^n(\Omega(w_2))\ge\calL^n(\Omega(w))
\text{ for all } w\ne w_1\,.
\end{equation*}
Since $\Z^N\cap B_{M+2N}$ contains a finite number of points we also have that
\begin{equation*}
\calL^n(\Omega)\le c \calL^n(\Omega(w_1))\,,\hskip1cm
\calL^n(\Omega\setminus\Omega(w_1))\le c \calL^n(\Omega(w_2))\,,
\end{equation*}
with $c$ depending only on $M$ and $N$.
Let $\lambda_1$ and $\lambda_2$ be the average of $\lambda(x\cdot\nu)$ over
$\Omega(w_1)$ and $\Omega(w_2)$ respectively. 
Then
\begin{equation*}
  \calL^n(\Omega(w_1)) |a\lambda_1 + b-w_1|\le 
  \|a\lambda(x\cdot\nu)+b-w\|_{L^1(\Omega(w_1))} \le \eta\,,
\end{equation*}
and since $\calL^n(\Omega)\le c \calL^n(\Omega(w_1))$, we
obtain
\begin{equation}\label{eqomega1abla1w1}
\calL^n(\Omega)|a\lambda_1 + b-w_1|\le 
c\eta\,.
\end{equation}
We set $b^*=w_1$.  Argueing as above we obtain
\begin{equation*}
\calL^n(\Omega\setminus\Omega(w_1))|a\lambda_2 + b-w_2|\le c\eta\,,
\end{equation*}
which implies
\begin{equation}\label{cstar}
\calL^n(\Omega\setminus\Omega(w_1))|a(\lambda_2-\lambda_1) -(w_2-w_1)|\le c^*\eta\,.
\end{equation}
If $\calL^n(\Omega\setminus\Omega(w_1))\le 2 c^*\eta$ then setting 
$a^*=\lambda^*=0$ will do. Otherwise, since $|w_2-w_1|\ge 1$, by (\ref{cstar}) we
obtain that $|a(\lambda_2-\lambda_1)|=|\lambda_2-\lambda_1|\ge 1/2$. 
Let 
\begin{equation*}
\xi = \min\{t>0: t(w_2-w_1)\in\Z^N\setminus\{0\}\}\,,
\end{equation*}
clearly $\xi\in[1/(2M+4N),1]$.
We set
\begin{equation*}
a^*=\xi(w_2-w_1)\in \Z^N\,,\hskip1cm
\tilde \lambda = \frac{\lambda-\lambda_1}{(\lambda_2-\lambda_1)\xi}\,. 
\end{equation*}
Then
\begin{equation*}
|(a^* \tilde\lambda +w_1)-( a \lambda +b)|\le
\left| (w_2-w_1) \frac{\lambda-\lambda_1}{\lambda_2-\lambda_1}  -
a (\lambda-\lambda_1)\right| + |w_1-(a \lambda_1 +b)|\,.
\end{equation*}
The second term can be controlled by (\ref{eqomega1abla1w1}).
The first one is bounded by $2 |\lambda-\lambda_1|\,
|(w_2-w_1)-(\lambda_2-\lambda_1) a|$. 
Integrating separately over $\Omega(w_1)$ and over
$\Omega\setminus\Omega(w_1)$,  using the estimate
$\|\lambda-\lambda_1\|_{L^1(\Omega(w_1))}\le c\eta$ and (\ref{cstar}), we obtain
\begin{equation*}
\|( a^* \tilde\lambda+w_1)-(a \lambda  +b)\|_{L^1(\Omega)}\le c\eta\,,
\end{equation*}
and recalling the definition of $w$
\begin{equation}\label{astartildez}
  \sum_w \| a^*\tilde \lambda +w_1-w\|_{L^1(\Omega(w))}\le c\eta\,.
\end{equation}
It remains to replace $\tilde\lambda$ by an integer-valued function $\lambda^*$. To do
this, consider
\begin{equation*}
\zeta = \inf \{\dist( \R z, \Z^N \cap B_{2M+4N} \setminus \R z) : z\in
\Z^N \cap B_{2M+4N}\}\,.
\end{equation*}
We remark that $\zeta>0$. Indeed, if this was not the case there would
be sequences $z_i, w_i\in \Z^N\cap B_{2M+4N}$ and $t_i\in \R$ such that
$|t_i z_i - w_i|\to 0$ and $w_i \not\in \R z_i$. By compactness the
sequences $z_i$ and $w_i$ have a constant subsequence, hence we obtain 
$|t_i z - w|\to0$. Since $\R z$ is closed this implies
$w\in \R z$, a contradiction. 

Fix one $w\in\Z^N\cap B_{M+2N}$.
If there is $\lambda_w\in \R$ such that 
$\lambda_w a^*=w-w_1$, then from the definition of $\xi$  we obtain
$\lambda_w\in\Z$, and we can set $\lambda^*=\lambda_w$ in $\Omega(w)$. 
Otherwise, $w-w_1\not\in \R a^*$, hence 
$| t  a^*  -w+w_1|\ge \zeta$ for all $t\in\R$.
In this case we set $\lambda^*=0$  in $\Omega(w)$, and estimate
\begin{equation*}
| w-w_1| \le 2M+4N \le \frac{2M+4N}{\zeta} |a^*\tilde\lambda + w_1-w|
\end{equation*}
pointwise in $\Omega(w)$, which gives
\begin{equation*}
  \|a^*\lambda^* + w_1-w\|_{L^1(\Omega(w))}
  \le  \frac{2M+4N}{\zeta}  
    \|a^*\tilde\lambda + w_1-w\|_{L^1(\Omega(w))}\,.
\end{equation*}
Recalling (\ref{eqdefwlambdaeta}) and (\ref{astartildez})  the proof is
concluded. 
\end{proof}

\begin{remark}\label{remZn}
The function $u^*$ constructed in the previous Lemma always satisfies
$$
\|u-u^*\|_{L^\infty}\leq C(M+N)\,.
$$
\end{remark}
%%Maybe it is not necessary to notice this, but it is used in the proof of Proposition~\ref{propconstruction}

We conclude this section with the following rigidity Lemma for affine
functions, which states 
that if an affine function on a square is close to the
set of integers, then it is close to a single integer. 
\begin{lemma}\label{lemmaaffineintegers}
There is a constant $\delta>0$ such that the
following holds: 
For every $Q=(-\ell,\ell)^2$, with $\ell>0$, every $A\in
\R^{N\times 2}$, $b\in 
\R^N$, if
\begin{equation}\label{eqdeltaaxbzn}
\frac{1}{\ell^2}\| \dist( Ax+b,\Z^N)\|_{L^1(Q)} \le \delta
\end{equation}
then there is $z\in \Z^N$ such that
\begin{equation*}
\| Ax+b-z\|_{L^1(Q;\R^N)} = \| \dist( Ax+b,\Z^N)\|_{L^1(Q)}\,.
\end{equation*}
\end{lemma}
\begin{proof}
Let $w:Q\to\Z^N$ be such that $\dist( Ax+b,\Z^N)=|Ax-b-w|$
pointwise. We claim that for an appropriate $\delta$ the condition
(\ref{eqdeltaaxbzn}) implies that $w$ is constant. 
To prove this, it suffices to show that any component is constant. Since 
$\dist((Ax+b)_i, \Z)\le 
\dist(Ax+b, \Z^N)$, it suffices to consider the case $N=1$.

Assume that $w$ is not constant. Then there is $\bar x\in Q$ such that
$|A\bar x+b-w|=1/2$. Since $x\mapsto \dist(Ax+b,\Z)$ is $|A|$-Lipschitz, we have
\begin{equation*}
\dist(Ay+b,\Z)\ge \frac12 - |A| \, |\bar x-y|\,.
\end{equation*}
Let $r=1/(4|A|\sqrt2)$, and assume $\bar x\in (0,\ell/2)^2$ (otherwise a
few signs have to be changed). 
On $\bar x+(0,r)^2$ we have
$\dist(Ax+b,\Z)\ge 1/4$. Now if $r\ge \ell/4\sqrt 2$ we have 
\begin{equation*}
\int_{Q} \dist(Ax+b,\Z)\, dx \ge
\int_{\bar x+(0,\ell/4\sqrt 2)^2} \dist(Ax+b,\Z)\, dx \ge
\frac{\ell^2}{32} \frac{1}{4}
\end{equation*}
and the proof is concluded (with $\delta< 1/128$).

Otherwise, set $R=1/|A|=4\sqrt2 r<\ell$. Choose at least
$ \frac{1}{4}\ell^2/R^2$ 
disjoint squares of side $2R$ contained in $Q$. Let $q=y+(-R,R)^2$ be one of them. 
Since $|A|R=1$, there is $\bar x\in y+(-R,0)^2$ such that
$\dist(A\bar x+b,\Z)=1/2$. Since $r=\frac{R}{4\sqrt 2}$, argueing as above, we obtain
\begin{equation*}
\int_{q} \dist(Ax+b,\Z)\, dx \ge \int_{\bar x+(0,R/4\sqrt 2)^2} \dist(Ax+b,\Z)\, dx \ge \frac{R^2}{32} \frac{1}{4}\,.
\end{equation*}
Summing over all squares the thesis follows with $\delta$ reduced by $1/4$.
\end{proof}

\section{Local approximation by one-dimensional functions}
\label{local1D}
Our next goal is to approximate functions with well-controlled energy by
one-dimensional functions. We first state the result and then explain the
meaning of the different quantities involved in the statement. 
Here and below we use the euclidean norm and scalar product for matrices,
i.e., $A\cdot B=\Tr A^TB$ and $|A|^2=\Tr A^TA$. For notational simplicity we
focus on  $W^{1,1}$ (resp. $L^1$) functions, the arguments in this section 
also hold in the case of $BV$ functions (resp. measures).

\begin{theorem}\label{propmakerankone}
Let $\ell>0$, $Q=(-\ell,\ell)^2$, $A\in \R^{N\times 2}$ with $|A|=1$,
$u\in W^{1,1}(Q;\R^N)$, and define 
\begin{alignat*}{1}
\eta_1 &= \frac{1}{\ell^2} \|\dist(u,\Z^N)\|_{L^1(Q)}\,,\\
\eta_2 &= \frac{1}{\ell} \|Du\|_{L^1(Q;\R^{N\times 2})}\,,\\
\eta_3 &= \frac{1}{\ell} \|Du-A(A\cdot Du)_+\|_{L^1(Q;\R^{N\times 2})}\,.
\end{alignat*}
Assume $\eta_1\le \delta/2$, $\delta$ being
as in Lemma \ref{lemmaaffineintegers}.
Then there are $a, b\in \R^N$,
$\nu\in S^1$, $\lambda\in
W^{1,1}(\R;\R)$, nondecreasing, such that the function
\begin{alignat*}{1}
\tilde u(x) = a\lambda (x\cdot\nu) +b
\end{alignat*}
obeys
\begin{alignat*}{1}
\frac{1}{\ell} \|u-\tilde u\|_{L^2(q;\R^N)} \le c
\eta_2^{2/3}\eta_3^{1/3}+c \eta_2\eta_3^{1/2}+c\eta_1
\end{alignat*}
and %, for some $\bar u\in\R^N$,
\begin{equation*}
\|a \lambda\|_{L^\infty(\R;\R^N)} \le c \eta_2\,.
\end{equation*}
Here $q=(-\ell/4,\ell/4)^2$, and the constant $c$ depends only on the
dimension $N$. 
\end{theorem} 
Here and below, $a_\pm=\max\{\pm a, 0\}$.

Note that the quantities $\eta_1$ and $\eta_2$ can be controlled by the
energy. In contrast the quantity $\eta_3$ is small whenever the $L^1$ norm of
a suitable mollification (on scale $l$) of $Du$ almost agrees with the $L^1$
norm of $Du$ 
(see Lemma \ref{lemma1dlocal} below).
We will see in Section~\ref{iterativemoll} that this property
holds for many scales. 

Before presenting the proof we discuss how this fundamental ingredient of our
construction can be made quantitative.
Since the $L^1$ norm is not strictly convex, the norm of a function
$f\in L^1(\R;[0,\infty))$  is the same as the norm of any
mollification, $\|f\|_{L^1(\R)}=\|f\ast\varphi\|_{L^1(\R)}$. The same,
however, does not hold for functions without a sign, or for vectorial
functions. The next lemma makes this quantitative, in a localized
way. We assert that if mollification does not decrease the $L^1$ norm
of a function substantially, then the function $f:\R^n\to\R^p$ is approximately
scalar, in the sense that there is a vector $\nu\in S^{p-1}$ such that $f$ is
close to the line $\nu[0,\infty)$. 

We shall apply this Lemma to the gradient of $u$, i.e., with $f=Du$,
$n=2$, $p=2N$, and the direction $\nu$ shall be an $N\times 2$ matrix.
\begin{lemma}\label{lemma1dlocal}\sloppypar
Let $f\in L^1_\loc(\R^n; \R^p)$, $\psi\in C_c(B_1,[0,\infty))$ 
be such that 
$\psi\ge 1$ on $B_{1/2}(0)$ and $\int_{B_1}\psi dx=1$,  and let  $\psi_r(x) = r^{-n}\psi(x/r)$.
Set $Q=(-r/2^{2+n/2},r/2^{2+n/2})^2$ and 
\begin{equation*}
\eta= \int_{\R^n} |f|
(\chi_{Q}\ast \psi_r)dx - \int_{Q} |f\ast \psi_r| dx \,.
\end{equation*}
Then the function $f$ is approximately
scalar, in the sense that there is $\nu\in S^{p-1}$ such that
\begin{equation*}
\int_{Q}( |f|-f\cdot \nu) dx \le c \eta
\end{equation*}
and
\begin{equation*}
\int_{Q} |f - \nu (f\cdot \nu)_+| dx \le c
\|f\|_{L^1(Q;\R^p)}^{1/2} \eta^{1/2} \,.
\end{equation*}
\end{lemma}
\begin{proof} By scaling we can assume $r=1$.
For $x\in Q$, 
let $\nu(x)\in S^{p-1}$ be a unit vector parallel to $(f\ast \psi)(x)$, so
that
\begin{equation}
|f\ast \psi|(x)=\int_{\R^n} f(y)\cdot \nu(x) \psi(x-y) \, dy\,.
\end{equation}
We define $\tilde \eta:Q\to[0,\infty)$ by 
\begin{alignat*}{1}
\tilde \eta(x) &= \int_{\R^n} |f|(y) \psi(x-y) dy - 
|f\ast \psi|(x) \\
&= 
\int_{\R^n} \left( |f|(y) - f(y)\cdot \nu(x)\right) \psi(x-y) dy\,.
\end{alignat*}
The integrand is obviously nonnegative.
Since for $x,y\in Q$ we have $|x-y|\le 1/2$, it follows
that
\begin{equation*}
\tilde \eta(x)\ge \int_{Q} \left( |f|(y) - f(y)\cdot \nu(x)\right)
\, dy\geq
\int_{Q} \left( |f|(y) - (f(y)\cdot \nu(x))_+\right) \, dy\,.
\end{equation*}
But by the definition of $\tilde\eta$ we obtain
\begin{equation*}
\int_Q \tilde \eta(x) \, dx = \eta\,.
\end{equation*}
Therefore there is
at least a point $x\in Q$ such that $\tilde\eta(x)\le 2^{n+2}\eta$. Setting
$\nu=\nu(x)$ concludes the proof of the first part. 

To prove the second part we observe that
\begin{equation*}
|f - \nu (f\cdot \nu)_+| = ( |f|^2 - (f\cdot \nu)_+^2)^{1/2} \le 2 |f|^{1/2}
(|f| -(f\cdot \nu)_+)^{1/2}\,. 
\end{equation*}
Using H\"older's inequality we obtain the thesis.
\end{proof}

We now prove that if the gradient of a function is one-dimensional, in the
sense that it can be well approximated by a scalar multiple of a fixed
matrix, then either the matrix is almost rank-one or the function is
almost affine. As usual in this kind of inequalities, when working in
$W^{1,p}$ with $1<p<\infty$ we can obtain full
control in the same space, whereas in the for us most
relevant case $p=1$ one can only estimate the function $u$ in the
corresponding space $L^{1^*}=L^2$.
\begin{proposition}\label{propkornpoincare}
Let $\Omega\subset\R^2$ be a bounded Lipschitz domain, $1<
p<\infty$. For any 
$u\in W^{1,p}(\Omega;\R^N)$,
$A\in \R^{N\times 2}$ with ${\rm rank}\,A=2$, $\xi\in L^{p}(\Omega;\R)$ there is $\bar
\xi\in\R$ 
such that 
\begin{equation*}
\|Du - \bar\xi A \|_{L^{p}(\Omega;\R^{N\times 2})} \le C \frac{a_1}{a_2} \|Du - \xi A\|_{L^p(\Omega;\R^{N\times 2})}\,.
\end{equation*}
Further, for any $u\in W^{1,1}(\Omega;\R^N)$,
$A\in \R^{N\times 2}$, $\xi\in L^{1}(\Omega;\R)$ there are $\bar
\xi\in\R$ and $b\in\R^N$ such that 
\begin{equation*}
\|u(x) - \bar\xi A x - b\|_{L^2(\Omega;\R^{N\times 2})} \le C
\frac{a_1}{a_2} 
\|Du - \xi A\|_{L^1(\Omega;\R^{N\times 2})}\,.
\end{equation*}
Here $a_1\ge a_2> 0$ are the singular values of $A$, i.e., the
eigenvalues of $(A^TA)^{1/2}$.
The constant depends on $p$, $\Omega$ and $N$. 
\end{proposition}
\begin{proof}
Set $\eta= \|Du - \xi A\|_{L^p(\Omega)}$ ($p=1$ in the second case).
By replacing $u$ with $\tilde u(x)=Qu(Rx)$, and $A$ with $\tilde
A=QAR$, with suitable $Q\in O(N)$, $R\in O(2)$, we
can assume $A$ to be diagonal, in the sense that $A=a_1e_1\otimes
e_1+a_2e_2\otimes e_2$, with $a_1\ge
a_2> 0$. For all $i=3, \dots N$ one has
\begin{equation*}
\|Du_i\|_{L^p}\le \eta
\end{equation*}
which implies the thesis for those components, hence it suffices to
treat the case $N=2$.
Define $v\in W^{1,p}(\Omega;\R^2)$ by 
\begin{equation*}
v_1(x) = a_1 u_2(x)\,,\hskip6mm
v_2(x) = -a_2 u_1(x)\,.
\end{equation*}
Then 
$(a_1 e_1\otimes e_2-a_2 e_2\otimes e_1)(Du-\xi A)=Dv - a_1a_2
\xi (e_1\otimes
e_2-e_2\otimes e_1)$, which implies
\begin{equation}\label{symmetric}
\left|\frac{Dv+Dv^T}{2}\right|\le \left|Dv - a_1a_2\xi (e_1\otimes
e_2-e_2\otimes e_1)\right| \le a_1 \left|Du -A\xi\right|
\end{equation}
pointwise. Therefore Korn's inequality shows that there is $\bar
\xi\in \R$ such that, for any $p>1$,
\begin{equation*}
\| Dv-a_1a_2\bar\xi (e_1\otimes
e_2-e_2\otimes e_1)\|_{L^p} \le C a_1 \left\|Du -
A\xi\right\|_{L^p}\,,
\end{equation*}
which in turn implies, using (\ref{symmetric}),
\begin{equation*}
a_1a_2 \|\xi-\bar\xi\|_{L^p} \le C a_1 \left\|Du -
A\xi\right\|_{L^p}\,.
\end{equation*}
Thus $\|\xi A-\bar\xi A\|_{L^p} \le C a_1 \|\xi-\bar\xi\|_{L^p} \le C
\frac{a_1}{a_2} \|Du-A\xi\|_{L^p}$, and the proof of the first part is concluded.

For $p=1$ the same estimates hold in weak-$L^1$, which does not embed
in $L^2$. However, from the Korn-Poincar\'e inequality (or the
embedding of $BD$ into $L^2$), one still has the existence of
$\bar\xi\in \R$ and 
$b\in \R^n$ such that
\begin{equation*}
\| v(x)-\bar\xi a_1a_2 (e_1\otimes
e_2-e_2\otimes e_1)x - b\|_{L^2} \le C a_1 \left\|Du -
A\xi\right\|_{L^1}\,,
\end{equation*}
which in turn implies
\begin{equation*}
\| u(x)-\bar\xi A x - b\|_{L^2} \le C \frac{a_1}{a_2} \left\|Du -
A\xi\right\|_{L^1}\,.
\end{equation*}
\end{proof}

Next we prove a Poincar\'e-type inequality where we only have half-sided
control on one component of the
gradient. We show that 
$u$ is close to an increasing function of one scalar variable alone.
\begin{lemma}\label{lemmapoincareoned}
Let $\ell>0$, $\nu\in S^1$, 
$Q_\nu^*=\{x: |x\cdot \nu|\le \ell, |x\cdot \nu^\perp|\le \ell/2\}$, $u\in
W^{1,1}(Q_\nu^*;\R)$. Set 
\begin{equation*}
\eta = \int_{Q_\nu^*} \left( |\partial_{\nu^\perp} u| +
|(\partial_\nu u)_-| \right) \, dx
\end{equation*}
and
$Q_\nu=\{x: |x\cdot \nu|\le \ell/2, |x\cdot \nu^\perp|\le \ell/2\}$
(see Figure \ref{fig:cogamu1}).
Then there is a nondecreasing function $h:\R\to\R$ such that
\begin{equation*}
\int_{Q_\nu} |u(x)-h(x\cdot \nu)|^2 \, dx \le C \eta |Du|(Q_\nu^*)
\end{equation*}
and 
\begin{equation*}
 \| h(x\cdot \nu) - \bar h\|_{L^\infty(Q_\nu)} \le \frac1\ell |Du|(Q_\nu^*)
\end{equation*}
for some $\bar h\in\R$.  
\end{lemma}
\begin{figure}[t]
\centering
\includegraphics[width=.6\textwidth]{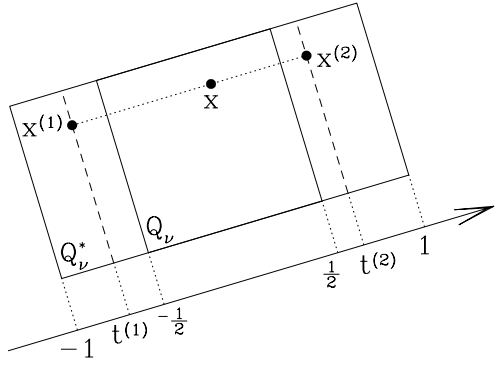}
\caption{Geometry in Lemma \ref{lemmapoincareoned}.}
\label{fig:cogamu1}
\end{figure}
\begin{proof}
By scaling we can assume $\ell=1$. Define $g:(-1,1)\to\R$ by
\begin{equation*}
g (t)= \int_{I(t)} u\, d\calH^1 
\end{equation*}
where $I(t) = Q_\nu^*\cap \{x: x\cdot \nu=t\}$, notice that
$\calH^1(I(t))=1$ for all $t\in(-1,1)$. For almost every $x\in
Q_\nu^*$ we have 
\begin{equation}\label{equglinfinity}
|u(x)-g(x\cdot \nu)| \le \int_{I(x\cdot \nu)} |\partial_{\nu^\perp} u|\,
d\calH^1 \,.
\end{equation}
Choose $t^{(1)}\in(-1,-1/2)$, $t^{(2)}\in(1/2,1)$ such that
\begin{equation}\label{eqsceltat1t2}
\int_{I(t^{(1)})\cup I(t^{(2)})} |\partial_{\nu^\perp} u| d\calH^1 \le 2\eta\,.
\end{equation}
We observe that $g\in W^{1,1}((-1,1))$, with
\begin{equation*}\label{eqgprimepoinc}
g'(t) = \int_{I(t)} \partial_\nu u \,d\calH^1\,,
\end{equation*}
which implies
\begin{equation*}
|g'_-|(t)\le \int_{I(t)} |(\partial_\nu u)_-|d\calH^1
\end{equation*}
for almost all $t\in(-1,1)$, and
\begin{equation*}
\int_{-1}^1 |g'| \, dt \le |Du|(Q_\nu^*)\,.
\end{equation*}
Therefore
\begin{equation*}
\int_{(-1,1)} |g'_-| (t) dt \le \eta\,.
\end{equation*}
For any $x\in Q_\nu$ we
set $x^{(2)}=t^{(2)}\nu + \nu^\perp (x\cdot\nu^\perp)= x +
(t^{(2)}-x\cdot \nu)\nu$, and estimate
\begin{alignat*}1
u(x)& = u(x^{(2)}) - \int_{[x,x^{(2)}]} \partial_\nu u\, d\calH^1\\
&\le g(t^{(2)}) + |u(x^{(2)})-g(t^{(2)})| + \int_{[x,x^{(2)}]}
|(\partial_\nu u)_-| \, d\calH^1 \,.
\end{alignat*}
As above, $[a,b]$ is the segment with endpoints $a$ and $b$. From
\begin{equation}\label{eqglinfty}
|g(x\cdot \nu)-g(t^{(2)})|\le \int_{[t_1,t_2]} |g'| dt \le |Du|(Q^*_\nu)\,,
\end{equation}
 (\ref{equglinfinity}), and (\ref{eqsceltat1t2}) we obtain
\begin{equation*}
|u(x^{(2)})-g(t^{(2)})|\le 2\eta\le 2 |Du|(Q^*_\nu)\,,
\end{equation*}
and therefore
\begin{alignat*}1
u(x)& \le g(x\cdot \nu) + 3 |Du|(Q^*_\nu)
+ \int_{[x,x^{(2)}]}
|(\partial_\nu u)_-| \, d\calH^1 \,.
\end{alignat*}
Analogously
\begin{alignat*}1
u(x)& \ge g(t^{(1)}) - |u(x^{(1)})-g(t^{(1)})| - \int_{[x^{(1)},x]}
|(\partial_\nu u)_-| d\calH^1 
\end{alignat*}
gives
\begin{alignat*}1
u(x)& \ge g(x\cdot \nu) - 3 |Du|(Q^*_\nu)
- \int_{[x^{(1)},x]}
|(\partial_\nu u)_-| \, d\calH^1 \,.
\end{alignat*}
We conclude that
\begin{equation*}
|u(x)-g(x\cdot\nu)| \le 3|Du|(Q^*_\nu) + \int_{[x^{(1)},x^{(2)}]}
|(\partial_\nu u)_-| 
d\calH^1 \,.
\end{equation*}
We multiply by (\ref{equglinfinity}) and integrate over $Q_\nu$, to obtain
\begin{equation*}
\int_{Q_\nu} |u(x)-g(x\cdot\nu)|^2 \, dx \le \left( \int_{Q_\nu^*}
|\partial_{\nu^\perp} u| dx \right) \left( 3|Du|(Q^*_\nu) + \int_{Q_\nu^*} |(\partial_{\nu} u)_-| dx \right) \,.
\end{equation*}
The second factor in the right-hand side can be controlled by $4|Du|(Q_\nu^*)$. 

Finally, we define $h:\R\to\R$ by $h(0)=g(0)$, $h' = g'_+$, and observe that
\begin{equation*}
\|h-g\|_{L^\infty(-1/2,1/2)}\le \int_{[-1,1]} |g'_-| dt \le \eta\,,
\end{equation*}
which concludes the proof of the first inequality. The uniform bound
follows from the definition of $h$ and 
 (\ref{eqgprimepoinc}).
\end{proof}

\begin{proof}[Proof of Theorem~\ref{propmakerankone}] Recall that
\begin{alignat*}{1}
\eta_1 &= \frac{1}{\ell^2} \|\dist(u,\Z^N)\|_{L^1(Q)}\,,\\
\eta_2 &= \frac{1}{\ell} \|Du\|_{L^1(Q)}\,,\\
\eta_3 &= \frac{1}{\ell} \|Du-A(A\cdot Du)_+\|_{L^1(Q)}\,,
\end{alignat*}
with $\eta_1\le \delta/2$, $\delta$ being
as in Lemma \ref{lemmaaffineintegers}, and that we have to show that there exists a function $\tilde u(x)=a\lambda (x\cdot\nu) +b$ (with $a, b\in \R^N$,
$\nu\in S^1$, $\lambda\in
W^{1,1}(\R;\R)$), such that
\begin{alignat*}{1}
\frac{1}{\ell} \|u-\tilde u\|_{L^2(q)} \le c
\eta_2^{2/3}\eta_3^{1/3}+c \eta_2\eta_3^{1/2}+c\eta_1
\end{alignat*}
and %, for some $\bar u\in\R^N$,
\begin{equation*}
\|a \lambda\|_{L^\infty(\R)} \le c \eta_2\,.
\end{equation*}

By scaling we can assume $\ell=1$; from $|A|=1$ one obtains 
$\eta_3\le\eta_2$.
Let $a_1\ge a_2\ge0$ be the singular values of $A$. 

The argument is based on obtaining two different estimates, and then choosing
one or the other depending on the value of $a_2$ relative to the $\eta_{1,2,3}$.

{\bf Step 1.}
Assume first $a_2>0$, i.e., $\rank A=2$.
Setting
$\xi=(A\cdot Du)_+$, from Proposition
\ref{propkornpoincare} and $|A|=1$ we have
\begin{equation*}
\|u(x) - \bar\xi A x - b\|_{L^2(Q)} \le C \frac{a_1}{a_2} \eta_3
\le C \frac{1}{a_2} \eta_3\,,
\end{equation*}
for some $\bar\xi\in \R$, $b\in \R^N$.
This implies 
\begin{equation*}
\|\dist(\bar\xi A x + b,\Z^N)\|_{L^1(Q)} \le c_* \frac{1}{a_2}
\eta_3+\eta_1\,. 
\end{equation*}
We distinguish two cases. If $c_*\eta_3\le \delta a_2/2$, then 
the right-hand side is less then $\delta$, and
by Lemma \ref{lemmaaffineintegers} there is $z\in \Z^N$ such that
\begin{equation*}
\|\bar\xi Ax -b -z\|_{L^1(Q)} \le c_* \frac{1}{a_2} \eta_3+\eta_1\,.
\end{equation*}
This immediately implies
\begin{equation*}
|\bar\xi A| \le C (\frac{1}{a_2} \eta_3+\eta_1)\,.
\end{equation*}
We conclude that at least one of the two inequalities
\begin{equation}\label{equmenobcaseone}
\|u-b\|_{L^2(Q)}\le C \frac{1}{a_2} \eta_3+C \eta_1
\end{equation}
or
\begin{equation}\label{eqa2upperbound}
a_2\le C'\eta_3
\end{equation}
holds. Here both constants may only depend on $N$.
It is clear that the same conclusion holds also in the 
remaining case $a_2=0$.

\begin{figure}[t]
\centering
\includegraphics[width=.6\textwidth]{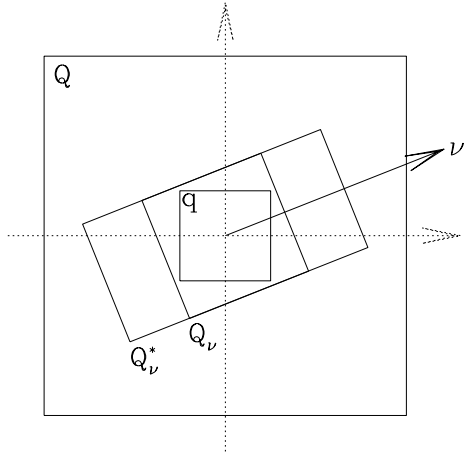}
\caption{Geometry in Step 2 of the proof of Proposition \ref{propmakerankone}.}
\label{fig:cogamu2}
\end{figure}

{\bf Step 2.}
Choose $\alpha,\alpha'\in S^{N-1}$, $\nu,\nu'\in S^1$ orthogonal and so that
$A=a_1\alpha\otimes \nu + a_2 \alpha'\otimes\nu'$, then
\begin{equation*}
|A - \alpha\otimes \nu|\le |1-a_1|+|a_2|\le 2a_2\,.
\end{equation*}
Writing $B=\alpha\otimes\nu$, we obtain
\begin{alignat*}{1}
|A(A\cdot Du)_+-B(B\cdot Du)_+|&\le |A-B|\, |(B\cdot Du)_+|+ |A|\,
|(A\cdot Du)_+-(B\cdot Du)_+|\\ 
& \le 2 |A-B| \, |Du|\,,
\end{alignat*}
which implies
\begin{alignat}1
\|Du -& \alpha\otimes \nu (\partial_\nu u\cdot \alpha)_+\|_{L^1(Q)}
\nonumber\\
&\le \|Du - A (A\cdot Du)_+\|_{L^1(Q)}+
\|A(A\cdot Du)_+-B(B\cdot Du)_+\|_{L^1(Q)} \nonumber\\
& \le
\eta_3+4{a_2} \eta_2\,.  \label{eqDuRankone}
\end{alignat}
Let $P_\alpha^\perp=\Id_N-\alpha\otimes \alpha$ be the projection on the space orthogonal
to $\alpha$. Since $P_\alpha^\perp(\alpha\otimes \nu)=0$ we deduce 
\begin{equation*}
\|D (P_\alpha^\perp u)\|_{L^1(Q)}\le \eta_3+4a_2\eta_2\,.
\end{equation*}
Therefore there is $b\in\R^N$, $b\cdot \alpha=0$, such that
\begin{equation}\label{eqPalphaperpub}
\|P_\alpha^\perp u-b\|_{L^2(Q)}\le c
\|D P_\alpha^\perp u\|_{L^1(Q)}
\le c\eta_3+c{a_2}\eta_2\,.
\end{equation}
The component $u\cdot \alpha$ is treated using Lemma \ref{lemmapoincareoned}. Indeed, with
the notation in that statement (using the present $\nu$, $\ell=1/\sqrt2$) we have
$q\subset Q_\nu\subset Q_\nu^*\subset Q$ (see Figure \ref{fig:cogamu2}). We
conclude together with (\ref{eqDuRankone}) that there is a 
monotone function $h$ such that
\begin{equation*}
\|(u\cdot \alpha)(x)-h(x\cdot \nu)\|_{L^2(q)} \le C ( \eta_3 + 
a_2\eta_2)^{1/2}  \eta_2^{1/2}\,.
\end{equation*}
Combining this with (\ref{eqPalphaperpub}) and dropping irrelevant terms we obtain 
\begin{equation}\label{eqfinalsteptwoumenoahb}
\|u(x)-\alpha h(x\cdot \nu)-b\|_{L^2(q)} \le 
c a_2^{1/2} \eta_2+
c \eta_2^{1/2}\eta_3^{1/2}\,. 
\end{equation}

We finally come back to the two cases we distinguished at the end of
Step 1. If (\ref{eqa2upperbound}) holds, then
(\ref{eqfinalsteptwoumenoahb}) becomes
\begin{equation*}
\|u(x)-\alpha h(x\cdot \nu)-b\|_{L^2(q)} \le 
c \eta_2 \eta_3^{1/2}+
c \eta_2^{1/2}\eta_3^{1/2}\,. 
\end{equation*}
In this case the proof is concluded.

Assume now that (\ref{eqa2upperbound}) does not hold. 
If $\frac{1}{a_2} \eta_3+\eta_1>  a_2^{1/2} \eta_2+
\eta_2^{1/2}\eta_3^{1/2}$ we set  $\tilde u(x)=\alpha h(x\cdot \nu)+b$, otherwise
$\tilde u=b$. From (\ref{equmenobcaseone}) and 
(\ref{eqfinalsteptwoumenoahb}) we then obtain 
\begin{alignat*}{1}
\|u-\tilde u\|_{L^2(q)}\le& c \min\left\{ 
\frac{1}{a_2} \eta_3+\eta_1,
a_2^{1/2} \eta_2+
\eta_2^{1/2}\eta_3^{1/2} \right\}\\
\le&
c \min\left\{ 
\frac{\eta_3}{a_2},
a_2^{1/2} \eta_2 \right\}
+c \eta_1+c
\eta_2^{1/2}\eta_3^{1/2}\,.
\end{alignat*}
But $\min\{\eta_3/a_2, a_2^{1/2}\eta_2\}\le \eta_3^{1/3}\eta_2^{2/3}$,
and we conclude
\begin{equation*}
\|u-\tilde u\|_{L^2(q)}\le c
\eta_2^{2/3}\eta_3^{1/3}+c \eta_1\,.
\end{equation*}
\end{proof}

\section{Control of the line energy with the truncated energy}
\label{control}

In the previous section we saw that functions with low energy (and small
differences $\|Du\|_{L^1}-\|\varphi * Du\|_{L^1}$) are well approximated by
one-dimensional functions. In Section~\ref{1D} we saw that for one-dimensional
functions the truncated energy is well approximated by the line energy. Now we
combine these results to obtain a global approximation: given a function $u\in
BV(\Omega; \R^N)$ and $\omega\subset\subset \Omega$ we construct a new
function $w\in BV(\omega;\Z^N)$ such that the relaxed line energy of $w$ 
\begin{equation}
E_0^\rel[w,\omega] = \int_{J_w\cap \omega} \gamma_0^\rel(\nu,[w]) \, d\calH^1
\end{equation}
is essentially controlled by the truncated energy of $u$
(we switch from $\gamma_0$ to the smaller $\gamma_0^\rel$, which has linear
growth, since  boundary terms are only controlled in $L^1$). 
We fix a mollifier $\varphi\in C_c(B_1;[0,\infty))$, with $\int_{\R^2}
\varphi\,dx=1$ and $\varphi\ge 1$ on $B_{1/2}(0)$. Let
$\varphi_h(x)=2^{2h}\varphi(2^hx)$.  

\begin{proposition}\label{propconstruction}
Let $\omega\subset\subset\Omega$ be two Lipschitz sets, 
$u\in W^{1,1}(\Omega;\R^N)$,
$M>1$, $h,t\in\N$ with $t\ge 3$, $\eta\in(0,1)$.
Assume $\dist(\omega,\partial\Omega)\ge 2^{-h+1}$.

Then there is $w=w_{M,h,t,\eta}\in BV(\omega;\Z^N)$ such that
\begin{alignat}{1}
(\ln 2) \int_{J_w\cap \omega} \gamma_0^\rel(\nu,[w]) \, d\calH^1 \le& 
(1+\eta+c 2^{-t}) p_{\Gamma_{h+t}, \Omega}(u) 
+\frac{C_M}\eta 2^{h+t} \|\dist(u,\Z^N)\|_{L^1(\Omega)}
\nonumber
\\
&+\frac{C_M}\eta 2^{t} A^{5/6} 
\left(|Du|(\Omega) - |D(u \ast\varphi_{h})|(\omega)\right)^{1/6}
\nonumber
\\ 
&+ \frac{c}{M^{1/2}} 2^{t/2} A
\label{eqpropconstrlongestimateenergy}
\end{alignat}
and
\begin{alignat}{1}
\|u-w\|_{L^1(\omega)}\le& \frac{c}{M^{1/2}} 2^{-h+t/2} A
+ C_M \|\dist(u,\Z^N)\|_{L^1(\Omega)}
\nonumber
\\
&
+ C_M 2^{-h} A^{2/3} 
\left(|Du|(\Omega) - |D(u \ast\varphi_{h})|(\omega)\right)^{1/3}\,.
\label{eqpropconstrlongestimateLone}
\end{alignat}
Here $A=\max\{|Du|(\Omega), p_{h+t,\Omega}(u)\}$.
\end{proposition}
\begin{proof}
{\bfseries Step 1. Domain subdivision.}
For $z\in \Z^2$, define $Q_{z}^*=(z+[-1,1]^2)2^{-h-5}$ and
$Q_{z}^{**}=(z+[-4,4]^2)2^{-h-5}$. We shall consider those $z$ for
which the larger square touches $\omega$, i.e., those in 
\begin{equation*}
Z=\{z\in \Z^2: Q_z^{**}\cap \omega\ne\emptyset\}\,.
\end{equation*}
The larger squares have finite overlap, are contained in $\Omega$, and give a uniform cover of
$\omega$, in the sense that 
\begin{equation}
\label{eq:chiqzchiqzstst}
64 \chi_\omega \le \sum_{z\in Z} \chi_{Q_z^{**}} \le 64 \chi_\Omega
\hskip1cm\text{a.e.}\,.
\end{equation}
The set $\omega$ is covered by the smaller squares, in
the sense that
\begin{equation}
\label{eq:chiqzchiqzststdue}
4 \chi_\omega \le \sum_{z\in Z} \chi_{Q_z^{*}} \,.
\end{equation}
We assert that we can find 
for each $z\in Z$ a function $u_z^*\in BV(Q_z^{*},\Z^N)$
such that
\begin{alignat}{1}\nonumber
\sum_{z\in Z} \| u-u_z^*\|_{L^1(Q_z^*)} \le &
\frac{c}{M^{1/2}} 2^{-h+t/2} A
+ C_M \|\dist(u,\Z^N)\|_{L^1(\Omega)}
\\
&
+ C_M 2^{-h} A^{2/3} 
\left(|Du|(\Omega) - |D(u \ast\varphi_{h})|(\omega)\right)^{1/3}
\,.
\label{eq:l1distallqz}
\end{alignat}
Further, for some $B\subset Z$,
$u_z^*$ is constant if $z\in B$, and
\begin{alignat}{1}\nonumber
\sum_{z\in Z\setminus B} \| u-u_z^*\|_{L^2(Q_z^*)}^2 &\le 
C_M \|\dist(u,\Z^N)\|_{L^1(\Omega)}
\\
&
+ C_M 2^{-h} A^{5/6} 
\left(|Du|(\Omega) - |D(u \ast\varphi_{h})|(\omega)\right)^{1/6}
\,.
\label{eq:l2distallqz}
\end{alignat}
We treat in Step 2 the squares in $B$, in Step 3 those in $Z\setminus
B$. We start by defining the set of ``bad'' squares by
\begin{equation}
B = \{ z\in Z: |Du|(Q_z^{**}) \ge M 2^{-h} \}\,,
\end{equation}
and at the same time the set of ``good'' squares
\begin{equation}
G = Z\setminus B = \{ z\in Z: |Du|(Q_z^{**}) < M 2^{-h} \}\,.
\end{equation}

{\bfseries Step 2. The ``bad'' squares.}
This argument is similar to the one used in proving Proposition
\ref{propmakeubv}. Fix one $z\in B$.
We shall  subdivide $Q_z^*$ into smaller
squares of side  $\alpha=2^{-(h+t+2)}$, and show that
$u$ does not change much from small square to small square. This will 
allow us to replace $u$ by a constant in the entire square
$Q_z^*$.

Precisely, for $\zeta\in \Z^2$ we set
$q_\zeta = \alpha \zeta +(0,\alpha)^2$, $Q_\zeta=\alpha \zeta + (-\alpha,
2\alpha)^2$, and $W_z = \{\zeta\in \Z^2 : Q_\zeta\subset Q_z^{**}\}$.
We define $\hat u_z: W_z\to\R^N$ by setting $\hat u_z(\zeta)$ equal to the
average of  $u$ over the square $Q_\zeta$. Reasoning as in (\ref{eqpgammakQz})
of Proposition \ref{propmakeubv} we obtain 
\begin{equation*}
\int_{Q_\zeta} |u-\hat u_z(\zeta)|^2 \, dx\le c 2^{-h-t} p_{h+t,Q_\zeta}(u)\,.
\end{equation*}
For $|\zeta-\zeta'|=1$ we have $\calL^2(Q_\zeta\cap Q_{\zeta'})\ge \alpha^2$, and therefore
\begin{equation*}
\sum_{\zeta,\zeta'\in W_z: |\zeta-\zeta'|=1} \alpha^2 |\hat u_z(\zeta)-\hat
u_z(\zeta')|^2 \le c 2^{-h-t} p_{h+t,Q_z^{**}}(u)\,.
\end{equation*}
Since $W_z$ is a discrete square, the discrete Poincar\'e inequality yields a
$v_z\in\R^N$ such that 
\begin{equation*}
\sum_{\zeta\in W_z} \alpha^4 |\hat u_z(\zeta)-v_z|^2 \le c 2^{-2h} 2^{-h-t} p_{h+t,Q_z^{**}}(u)\,.
\end{equation*}
Choose $u_z^*\in\Z^N$ such that $|u_z^*-v_z|\le N$. Then
\begin{alignat*}{1}
\|u-u_z^*\|^2_{L^2(Q_z^{**})} \, dx &\le
3\sum_{\zeta\in W_z} \left[ \int_{q_\zeta} |u-\hat u_z(\zeta)|^2 \, dx + \alpha^2 |\hat
u_z(\zeta)-v_z|^2 +\alpha^2 N^2 \right]\\
&\le c 2^{-h-t} p_{h+t,Q_z^{**}}(u)
+c \alpha^{-2} 2^{-2h} 2^{-h-t} p_{h+t,Q_z^{**}}(u) + c  2^{-2h}\\
&\le c 2^{-h+t} p_{h+t,Q_z^{**}}(u) + c  2^{-2h}\,,
\end{alignat*}
and since $z\in B$, 
\begin{alignat*}{1}
\|u-u_z^*\|_{L^1(Q_z^*)}& \le c 2^{-h}\|u-u_z^*\|_{L^2(Q_z^*)}\\
&\le c 2^{-h+t/2}2^{-h/2} \left( p_{h+t,Q_z^{**}}(u)\right)^{1/2} + c  2^{-2h}  \\
&\le \frac{c}{M^{1/2}} 2^{-h+t/2} \left( |Du|(Q_z^{**})\right)^{1/2}
\left( p_{h+t,Q_z^{**}}(u)\right)^{1/2} + \frac{c}{M}2^{-h}  |Du|(Q_z^{**})\,. 
\end{alignat*}
We conclude that
\begin{alignat}{1}\label{estimatebad}
\sum_{z\in B} \|u-u_z^*\|_{L^1(Q_z^*)}
&\le \frac{c}{M^{1/2}} 2^{-h+t/2} \left( |Du|(\Omega)\right)^{1/2}
\left( p_{h+t,\Omega}(u)\right)^{1/2} + \frac{c}{M} 2^{-h}  |Du|(\Omega)\nonumber\\
&\le \frac{c}{M^{1/2}} 2^{-h+t/2} \max\left\{ |Du|(\Omega),
p_{h+t,\Omega}(u)\right\}\,. 
\end{alignat}
This proves (\ref{eq:l1distallqz}) for the ``bad'' squares.

{\bfseries Step 3. The ``good'' squares.}
Let $z\in Z$. We apply Lemma \ref{lemma1dlocal}
to $f=Du$ on the square $Q_{z}^{**}$, with $r=2^{-h}$, and the mollifier
$\varphi_h$ (it is here important that the side of
$Q_z^{**}$ is $2^{-h-2}$). For each of them we obtain a 
matrix $A_z\in\R^{N\times 2}$ such that the quantity
\begin{equation*}
\eta^z_3=2^h \|Du-A_z(A_z\cdot Du)_+\|_{L^1(Q_{z}^{**})}
\end{equation*}
obeys
\begin{equation}\label{eqboundetaz3}
\eta^z_3\le c 2^{h} \left(
\int_{\R^2} |Du| (\varphi_h\ast \chi_{Q_{z}^{**}})
- \int_{\R^2} |D(u \ast\varphi_{h})|\chi_{Q_{z}^{**}} 
\right)^{1/2} 
\left( |Du|(Q_{z}^{**})\right)^{1/2}\,.
\end{equation}
We intend to apply Theorem \ref{propmakerankone} to the pair of
squares $q=Q_z^*\subset Q=Q_z^{**}$, with $\ell=2^{-h-3}$ and $z\in G$.
Therefore we define, analogously to Proposition
\ref{propmakerankone} (but, for notational convenience, without the
factors $2^{-3}$),
\begin{equation*}
\eta_1^z = 2^{2h} \|\dist(u,\Z^N)\|_{L^1(Q_z^{**})}
\end{equation*}
and
\begin{equation*}
\eta_2^z = 2^{h} |Du|(Q_z^{**})\,.
\end{equation*}
For the values of $z$ such that $\eta_1^z\le \delta/2^7$, i.e., 
\begin{equation}\label{eqasseta1}
\|\dist(u,\Z^N)\|_{L^1(Q_z^{**})}\le \delta 2^{-2h-7} \,,
\end{equation}
we can apply 
Theorem \ref{propmakerankone} to the square $Q_z^{**}$, 
and obtain
$\nu_z\in S^1$,
$a_z, b_z\in\R^N$, and a monotone function $\lambda_z$, such that
the function $\tilde u_z(x)=a_z\lambda_z(x\cdot\nu_z)+b_z$ obeys
\begin{alignat}{1}
\|u-\tilde u_z\|_{L^2(Q_{z}^*)} & =
\|u(x) - a_z \lambda_z(x\cdot\nu_z)-b_z\|_{L^2(Q_{z}^*)}\nonumber\\
& \le c
2^{-h} ( (\eta_2^z)^{2/3}(\eta_3^z)^{1/3} + \eta_2^z(\eta_3^z)^{1/2} +
\eta_1^z)\,,
\nonumber
\end{alignat}
with 
\begin{equation*}
\|a_z \lambda_z\|_{L^\infty(\R)} \le c \eta_2^z\,.
\end{equation*}
Since $z\in G$ we have $\eta_3^z\le\eta_2^z\le M$, and the above conditions
imply 
\begin{alignat}{1}
\|u-\tilde u_z\|_{L^2(Q_{z}^*)} \le&
C_M 
2^{-h} ( (\eta_2^z)^{2/3}(\eta_3^z)^{1/3} + \eta_1^z)\,,
\label{equtildeumollifgener}
\end{alignat}
and
\begin{equation*}
\|a_z \lambda_z\|_{L^\infty(\R)} \le c M\,.
\end{equation*}
Therefore we can apply Lemma \ref{lemmaZn} to the function $\tilde u_z$, and
obtain $a_z^*$, $b_z^*\in\Z^N$ and 
$\lambda_z^*\in L^1(\R;\Z)$ such that
$u^*_z(x)=a_z^*\lambda_z^*(x\cdot\nu_z^*)+b_z^*$ obeys
\begin{equation}\label{equstarmenoutile}
\|\tilde u_z-u_z^*\|_{L^1(Q_z^*)}\le C_M \|\dist(\tilde u_z, \Z^N)\|_{L^1(Q_z^*)}\,.
\end{equation}
Here and below the dependence of the constant on $M$ is indicated
explicitly, whereas we do not indicate the dependence on $N$.
In turn, using Remark~\ref{remZn}, (\ref{equstarmenoutile}) gives
\begin{alignat*}{1}
\|\tilde u_z-u^*_z\|_{L^2(Q_z^*)}^2\le& C (M+N) 
\|\tilde u_z-u^*_z\|_{L^1(Q_z^*)}\\
\le & C_M \|\dist(\tilde u_z, \Z^N)\|_{L^1(Q_z^*)}\\
\le & C_M \left( \|\dist( u, \Z^N)\|_{L^1(Q_z^*)} 
+ \| u-\tilde u_z\|_{L^1(Q_z^*)}\right)\,.
\end{alignat*}
Recalling the definition of $\eta_1^z$ and (\ref{equtildeumollifgener}), we obtain 
\begin{alignat*}{1}
\|\tilde u_z-u^*_z\|_{L^2(Q_z^*)}^2\le& C_M 
\left( 2^{-2h} \eta_1^z 
+ 2^{-h} \| u-\tilde u_z\|_{L^2(Q_z^*)}\right)\\
\le & C_M 2^{-2h}
\left( (\eta_2^z)^{2/3}(\eta_3^z)^{1/3} + \eta_1^z
\right)\,.
\end{alignat*}
Since we assumed $\eta_3^z\le \eta_2^z\le M$ and $\eta_1^z\le \delta/2^7$, 
estimate (\ref{equtildeumollifgener}) implies
\begin{alignat*}{1}
\|u-\tilde u_z\|_{L^2(Q_{z}^*)}^2& \le C_M
2^{-2h} ( (\eta_2^z)^{2/3}(\eta_3^z)^{1/3} +
\eta_1^z)\,.
\end{alignat*}
Therefore for those $z\in G$ for which (\ref{eqasseta1}) holds we have
\begin{alignat}{1} \nonumber
\|u-u_z^*\|_{L^2(Q_z^*)}^2
\le &2
\left(\|u-\tilde u_z\|_{L^2(Q_z^*)}^2+ \|\tilde u_z-u_z^*\|_{L^2(Q_z^*)}^2\right)\\
\le &C_M 2^{-2h} \left( (\eta_2^z)^{2/3}(\eta_3^z)^{1/3} + \eta_1^z
\right) \,.
\label{eqprimagammahtqzstst}
\end{alignat}
If instead $z\in G$ is such that (\ref{eqasseta1}) does not hold, then we
take $u^*_z$ constant, equal to the integer closest to the average of
$u$. Then  by the Sobolev-Poincar\`e inequality
\begin{equation*}
\|u-u^*_z\|_{L^2(Q_{z}^*)}^2\le c 2^{-2h} +c \left( |Du|(Q_z^*)\right)^2 \le
c (1+M^2) 2^{-2h} \le C_M \eta_1^z 2^{-2h}\,.
\end{equation*}
Therefore the estimate (\ref{eqprimagammahtqzstst}) holds for all $z\in G$.

We conclude that
\begin{alignat*}{1}
\sum_{z\in G} & \|u-u_z^*\|_{L^2(Q_z^*)}^2
\le C_M 2^{-2h} \left( \sum_{z\in G} (\eta_2^z)^{2/3}(\eta_3^z)^{1/3} +
\sum_{z\in G} \eta_1^z
\right)\\
&\le C_M 2^{-h} \left( \sum_{z\in G} 2^{-h}\eta_2^z \right)^{2/3}
\left( \sum_{z\in G} 2^{-h}\eta_3^z \right)^{1/3}
+ C_M  \sum_{z\in G} 2^{-2h}\eta_1^z
\,.
\end{alignat*}
Since $\sum\chi_{Q_z^{**}}\le C\chi_\Omega$, we have
\begin{alignat*}{1}
& \sum_{z\in G} \|u-u_z^*\|_{L^2(Q_z^*)}^2\\
&\le C_M 2^{-h} \left( |Du|(\Omega) \right)^{2/3}
\left( \sum_{z\in G} 2^{-h}\eta_3^z \right)^{1/3}
+ C_M \|\dist(u, \Z^N)\|_{L^1(\Omega)}
\,.
\end{alignat*}
The term containing $\eta_3^z$ is estimated using (\ref{eqboundetaz3}),
\begin{alignat*}{1}
& \sum_{z\in G} 2^{-h}\eta^z_3 \le\sum_{z\in Z} 2^{-h}\eta^z_3 \\
\le & c \left(
\int_{\R^2} |Du| \left[\varphi_h \ast \sum_{z\in Z} \chi_{Q_{z}^{**}} \right]
- \int_{\R^2} |D(u \ast\varphi_{h})|\sum_{z\in Z} \chi_{Q_{z}^{**}} 
\right)^{1/2}
\left( |Du|(\Omega)\right)^{1/2}\,.
\end{alignat*}
Recalling (\ref{eq:chiqzchiqzstst}), and the fact that
$\dist(Q_z^{**},\partial\Omega)\ge
\dist(\omega,\partial\Omega)-\diam(Q_z^{**})\ge
2^{-h}$ for all $z\in Z$, we obtain
\begin{equation*}
\varphi_h \ast \sum_{z\in Z} \chi_{Q_{z}^{**}} \le 
64 \varphi_h \ast \chi_{\cup \chi_{Q_{z}^{**}}} \le 
64\chi_\Omega
\end{equation*}
and $\sum_{z\in Z} \chi_{Q_{z}^{**}}\ge 64\chi_\omega$. Therefore
\begin{equation*}
\sum_{z\in G} 2^{-h}\eta^z_3 \le c
\left(|Du|(\Omega) - |D(u \ast\varphi_{h})|(\omega)\right)^{1/2}
\left( |Du|(\Omega)\right)^{1/2}\,.
\end{equation*}
We conclude
\begin{alignat*}{1}
& \sum_{z\in G} \|u-u_z^*\|_{L^2(Q_z^*)}^2\\
&\le C_M 2^{-h} \left( |Du|(\Omega) \right)^{5/6}
\left(|Du|(\Omega) - |D(u \ast\varphi_{h})|(\omega)\right)^{1/6}
+ C_M \|\dist(u, \Z^N)\|_{L^1(\Omega)}
\,.
\end{alignat*}
This concludes the proof of (\ref{eq:l2distallqz}).

Finally, from (\ref{equtildeumollifgener}) and (\ref{equstarmenoutile}) we
have
\begin{equation*}
\|u-u_z^*\|_{L^1(Q_z^*)}\le C_M \|\dist(\tilde u_z, \Z^N)\|_{L^1(Q_z^*)}
+ C_M 2^{-2h} ( (\eta_2^z)^{2/3}(\eta_3^z)^{1/3} + \eta_1^z)\,.
\end{equation*}
Estimating the sum over all squares as above, 
\begin{alignat*}{1}
\sum_{z\in G} \|u_z^*- u\|_{L^1(Q_z^*)}&\le C_M \|\dist(\tilde u_z,
\Z^N)\|_{L^1(\Omega)} \\
&+ C_M 2^{-h}
\left(|Du|(\Omega) - |D(u \ast\varphi_{h})|(\omega)\right)^{1/3} 
\left(|Du|(\Omega) \right)^{2/3} \,.
\end{alignat*}
This, together with (\ref{estimatebad}), concludes the proof of (\ref{eq:l1distallqz}).

{\bfseries Step 4. Global construction.}
Based on the functions constructed above on each square, which obey
the estimates (\ref{eq:l1distallqz}) and (\ref{eq:l2distallqz}), we
shall now construct the global function $w$. The first idea is to set
$w=u_z^*$ in $Q_z=(z+[0,1]^2)2^{-h-5}$. Since (\ref{eq:l1distallqz}) gives
only control of $u-u_z^*$ in $L^1$ the function $w$ could have large jumps on
$\partial Q_z$, and may not be in $BV(\omega)$. The standard device to avoid
this is to set $w=u_z^*$ on the shifted squares
$Q_{a,z}=(a+z+[-1/2,1/2]^2)2^{-h-5}$. Then one can use Fubini's theorem to
show that there exists an $a\in [-1/4,1/4]^2$ such that $w$ has good $BV$
bound, see (\ref{eqfchoosea}) and (\ref{eqfchooseabis}) below. Since
Lemma \ref{lemma1Dinterfacesseparate} 
 gives a control of the line energy in terms of a slightly
enlarged square we also introduce the squares
$\hat Q_{a,z}$. 

Precisely, for any $a\in [-1/4,1/4]^2$ and $z\in \Z^2$ we define
$Q_{a,z}=(a+z+[-1/2,1/2]^2)2^{-h-5}$ and
$\hat Q_{a,z}=(a+z+[-1/2-2^{-t},1/2+2^{-t}]^2)2^{-h-5}$. We observe
that $Q_{a,z}\subset\hat Q_{a,z}\subset Q_z^*$ for all admissible
$a$, $t$, $z$. Further,
\begin{equation*}
\chi_\omega \le \sum_{z\in Z} \chi_{Q_{a,z}} \le \chi_\Omega 
\hskip1cm \text{ a.e. }
\end{equation*}
for all admissbile choices of $a$.

We define
\begin{equation*}
f(a) = \sum_{z,z'\in Z} \int_{ \partial Q_{a,z}\cap \partial Q_{a,z'}} |u^*_z - u^*_{z'}|(x) d\calH^1(x)
\end{equation*}
and observe that, by Fubini's theorem,
\begin{equation}\label{60bis}
\int_{(-1/4,1/4)^2} f(a) da \le c 2^h 
\sum_{z,z'\in Z} \|u^*_z - u^*_{z'}\|_{L^1(Q_z^*\cap Q_{z'}^*)}
\le c 2^h 
\sum_{z\in Z} \|u - u^*_{z}\|_{L^1(Q_z^*)}
\,.
\end{equation}
In order to control the error done by enlarging the squares we define
analogously
\begin{equation*}
g(a) = \sum_{z\in Z} \left[ p_{\Gamma_{h+t},\hat Q_{a,z}}(u)
-p_{\Gamma_{h+t},Q_{a,z}}(u) \right]\,.
\end{equation*}
and claim that
\begin{equation}\label{eqestintg}
\int_{(-1/4,1/4)^2} g(a) da \le c 2^{-t} p_{h+t,\Omega}(u)\,. 
\end{equation}
To see this, we write
\begin{alignat*}{1}
p_{\Gamma_{h+t},\hat Q_{a,z}}(u) -p_{\Gamma_{h+t},Q_{a,z}}(u)
= & \int_{\R^2\times \R^2} (u(x)-u(y))\cdot \Gamma_{h+t}(x-y) (u(x)-u(y))\\
&\times
\left[ \chi_{\hat Q_{a,z}}(x)\chi_{\hat Q_{a,z}}(y) -
\chi_{ Q_{a,z}}(x)\chi_{ Q_{a,z}}(y)\right] \, dx dy
\end{alignat*}
and observe that
\begin{alignat*}{1}
\chi_{\hat Q_{a,z}}(x)& \chi_{\hat Q_{a,z}}(y) -
\chi_{ Q_{a,z}}(x)\chi_{ Q_{a,z}}(y) = \\
&\chi_{\hat Q_{a,z}}(x)\left[ \chi_{\hat Q_{a,z}}(y)-\chi_{Q_{a,z}}(y)\right]
+\left[ \chi_{\hat Q_{a,z}}(x)-\chi_{Q_{a,z}}(x)\right]\chi_{ Q_{a,z}}(y)\,.
\end{alignat*}
Focussing on the second term we note that 
$\chi_{Q_{a,z}}\le \chi_{Q_{0,z}^*}$ and
$\chi_{Q_{a,z}}(x)=\chi_{Q_{0,z}}(x-2^{-h-5}a)$. Therefore 
\begin{alignat*}{1}
& \int_{(-1/4,1/4)^2} 
\left[ \chi_{\hat Q_{a,z}}(x)-\chi_{Q_{a,z}}(x)\right]\chi_{ Q_{a,z}}(y) da\\
& \le \chi_{ Q_{z}^*}(y) 
\chi_{ Q_{z}^*}(x) 
\int_{\R^2} \chi_{\hat Q_{0,z}\setminus
Q_{0,z}}(x-2^{-h-5}a) \, da\\
& \le 2^{2h+10}\calL^2(\hat Q_{0,z}\setminus Q_{0,z}) 
\chi_{ Q_{z}^*}(x)
\chi_{ Q_{z}^*}(y) 
\le c 2^{-t} \chi_{ Q_{z}^*}(x) \chi_{ Q_{z}^*}(y) \,.
\end{alignat*}
An analogous estimate holds for the other term. We conclude that
\begin{alignat*}{1}
\sum_{z\in Z} \int_{(-1/4,1/4)^2}& \left[\chi_{\hat Q_{a,z}}(x)\chi_{\hat
Q_{a,z}}(y) - \chi_{ Q_{a,z}}(x)\chi_{ Q_{a,z}}(y)\right]\, da\\
& \le c 2^{-t} \sum_{z\in Z} \chi_{ Q_{z}^*}(x) \chi_{ Q_{z}^*}(y) \\
&\le c 2^{-t} \chi_{\Omega}(x) \chi_{\Omega}(y)
\end{alignat*}
and
\begin{alignat*}{1}
\int_{(-1/4,1/4)^2} g(a) da \le &
\int_{\R^2\times \R^2} (u(x)-u(y))\cdot \Gamma_{h+t}(x-y) (u(x)-u(y))\\
&\times c 2^{-t} \chi_{\Omega}(x) \chi_{\Omega}(y)
\\
=& c 2^{-t} p_{h+t,\Omega}(u)\,. 
\end{alignat*}
This concludes the proof of (\ref{eqestintg}).

By (\ref{60bis})  and (\ref{eqestintg}) there exists 
 $a\in (-1/4,1/4)^2$ such that
\begin{equation}\label{eqfchoosea}
f(a) \le c 2^h 
\sum_{z\in Z} \|u - u^*_{z}\|_{L^1(Q_z^*)}
\end{equation}
and
\begin{equation*}
g(a)\le c 2^{-t} p_{h+t,\Omega}(u)\,. 
\end{equation*}
We define
\begin{equation*}
w = \sum_{z\in Z} u^*_z \chi_{Q_{a,z}}\,.
\end{equation*}
Clearly $w\in BV(\Omega;\Z^N)$, and
(\ref{eqpropconstrlongestimateLone}) follows from
(\ref{eq:l1distallqz}).
In order to prove (\ref{eqpropconstrlongestimateenergy}) we first observe that
\begin{equation}\label{eqfchooseabis}
  |Dw|\left(\bigcup_{z\in Z} \partial Q_{a,z}\right) \le c f(a)\,.
\end{equation}
The fact that $\gamma_0^\rel$ is convex in the first argument and subadditive
in the second easily implies   $|\gamma_0^\rel(\nu,s)|\le C |s|$
(to see this, consider that 
$ \gamma_0^\rel(\nu, s)\le \sum_{i=1}^N 
\sum_{j=1}^2 |s_i||\nu_j| \gamma_0^\rel(e_j, e_i)$). Therefore
\begin{equation*}
  \int_{\omega\cap\bigcup_{z\in Z} \partial Q_{a,z}} \gamma_0^\rel(\nu, [w]) d\calH^1
  \le c f(a)\,.
\end{equation*}
By Lemma \ref{lemma1Dinterfacesseparate} we obtain
\begin{equation*}
(\ln 2) E_0[w,Q_{a,z}]\le p_{\Gamma_{h+t}, \hat Q_{a,z}}(u_z^*)\qquad
\forall \ z\in G\,. 
\end{equation*}
Since $E_0^\rel\le E_0$ and  $ E_0^\rel[w,Q_{a,z}]=0$ whenever 
$z\in B$,
\begin{alignat*}{1}
(\ln 2) E_0^\rel[w,\omega]&\le \sum_{z\in G}(\ln 2) E_0[w,Q_{a,z}]
+ c f(a)\\
&\le \sum_{z\in G}p_{\Gamma_{h+t}, \hat Q_{a,z}}(u_z^*)
+ c f(a)\,.
\end{alignat*}
The first term can be estimated by
\begin{alignat*}{1}
\sum_{z\in G}&p_{\Gamma_{h+t}, \hat Q_{a,z}}(u_z^*) \le
(1+\eta) \sum_{z\in G}p_{\Gamma_{h+t}, \hat Q_{a,z}}(u) 
+\left(1+\frac1\eta \right) \sum_{z\in G}p_{\Gamma_{h+t}, \hat Q_{a,z}}(u_z^*-u)\\ 
&\le
(1+\eta) \sum_{z\in G}p_{\Gamma_{h+t}, Q_{a,z}}(u) + (1+\eta) g(a) 
+\left(1+\frac1\eta \right) c 2^{h+t} \sum_{z\in G}\|u_z^*-u\|_{L^2(Q_z^*)}^2\\
&\le
(1+\eta+c 2^{-t}) \sum_{z\in G}p_{\Gamma_{h+t}, Q_{a,z}}(u) 
+\left(1+\frac1\eta \right)c 2^{h+t} \sum_{z\in G}\|u_z^*-u\|_{L^2(Q_z^*)}^2
\end{alignat*}
Recalling (\ref{eq:l1distallqz}), (\ref{eq:l2distallqz}) and
(\ref{eqfchoosea}) we obtain (\ref{eqpropconstrlongestimateenergy}). This
finishes the proof of Proposition \ref{propconstruction}.
\end{proof}

\section{Iterative mollification and conclusion of the proof}\label{iterativemoll}

We now prove the following key result.

\begin{proposition}\label{proplowerbound}
Let $\Omega\subset\R^2$ be a bounded Lipschitz domain, and assume $u_0\in
BV(\Omega;\Z^N)$. Then for any sequences $\e_i\to0$, $u_i\to u_0$ in
$L^1(\Omega;\R^N)$ and any Lipschitz domain $\omega\subset\subset\Omega$
there is a sequence $w_j\in BV(\omega;\Z^N)$ such that $w_j\to u_0$ in
$L^1(\omega;\R^N)$ and 
\begin{equation*}
\liminf_{j\to\infty} E_0^\rel[w_j,\omega] \le \liminf_{i\to\infty}
E_{\e_i}[u_i,\Omega] \,. 
\end{equation*}
\end{proposition}

This result directly implies  Theorem~\ref{theorem2}. 
\begin{proof}[Proof of   Theorem~\ref{theorem2}]
Since $E_0^\rel$
is lower semicontinuous, 
$$
E_0^{\rm rel}[u_0,\omega] \le \liminf_{j\to\infty} E_0^\rel[w_j,\omega] \le \liminf_{i\to\infty}
E_{\e_i}[u_i,\Omega]\,.
$$ 
The conclusion follows by considering an increasing sequence $\omega_k$ with
$\bigcup\omega_k=\Omega$.  
\end{proof}

As already mentioned in Section~\ref{outline}, Theorem~\ref{theorem2} yields the
lower bound in the proof of Theorem~\ref{theorem1}. The upper bound is instead 
obtained by a more standard argument which we recall in the next section.

\begin{proof}[Proof of Proposition~\ref{proplowerbound}]
We choose a Lipschitz set $\Omega'$ such that
$\omega\subset\subset\Omega'\subset\subset\Omega$. 
For any $\delta>0$, 
Proposition
\ref{eqseqBV} applied to the pair of sets
$\Omega'\subset\subset\Omega$ gives functions
$v_{k,\delta}\in BV(\Omega';\Z^N)$ such that
\begin{equation}\label{eqliminfkliminfi}
\liminf_{k\to\infty} \frac{1}{k} \sum_{h=0}^k
p_{\Gamma_h,\Omega'}(v_{k,\delta}) \le (\ln 2)(1+\delta) \liminf_{i\to\infty}
E_{\e_i}[u_i,\Omega] \,,
\end{equation}
with $\lim_{k\to\infty} \|v_{k,\delta}-u_0\|_{L^1(\Omega')}=0$ and we can also assume that 
\begin{equation*}
\frac{1}{k} \sum_{h=0}^k
p_{\Gamma_h,\Omega'}(v_{k,\delta}) + |Dv_{k,\delta}|(\Omega')\le A_\delta\quad \forall k\,
\end{equation*}
(since such a bound holds for the subsequence in $k$ that realizes the liminf in
(\ref{eqliminfkliminfi})). 
The quantity $A_\delta$ may depend both on $\delta$ and on the
original sequence $u_i$, but not on the parameters which will be chosen
below. 

We define, for $h\in\N$,
\begin{equation*}
\Omega_h=\{x\in \R^2: B_{2^{-h}}(x)\subset\Omega'\}\,.
\end{equation*}
For a fixed $m\ge 3$ we define iteratively, for all $h\in \N\cap[0,k+m]$,
the functions $u_{k,\delta,m,h}\in BV(\Omega_h;\R^N)$ by
\begin{equation*}
u_{k,\delta,m,h} = 
\begin{cases}
v_{k,\delta} & \text { if } h\ge k\,, \\
u_{k,\delta,m,h+m}\ast \varphi_{h}& \text{ else.}
\end{cases}
\end{equation*}
The mollifier $\varphi_h$ was defined at the beginning of Section
\ref{control}. From the definition of $u_{k,\delta,m,h}$ we obtain, dropping
the first three indices to simplify the notation,
\begin{alignat*}{1}
\|u_h-u_{h+m}\|_{L^1(\Omega_h)}=\|u_{h+m}-u_{h+m}*\varphi_h\|_{L^1(\Omega_h)}&\le C 2^{-h} |Du_{h+m}|(\Omega_{h+m}) \\
&\le C
2^{-h} |Dv_{k,\delta}|(\Omega')\le C 2^{-h} A_\delta \,, 
\end{alignat*}
which, summing the geometric iteration, gives
\begin{equation*}
\|u_{k,\delta,m,h} -v_{k,\delta}\|_{L^1(\Omega_h)} \le C 2^{-h}
|Dv_{k,\delta}|(\Omega) \le C 2^{-h} A_\delta\,. 
\end{equation*}
Recalling that $v_{k,\delta}$ has value in $\Z^N$ a.e. we also obtain
\begin{equation}\label{eqdistuhZN}
\|\dist(u_{k,\delta,m,h}, \Z^N)\|_{L^1(\Omega_h)}\le C 2^{-h}
|Dv_{k,\delta}|(\Omega)\le 
C 2^{-h} A_\delta \,. 
\end{equation}
We further observe that by summing the telescoping series we get 
\begin{alignat}{1}\nonumber
\sum_{h=0}^k \left[ |Du_{h+m}|(\Omega_{h+m}) - |Du_{h}|(\Omega_h)\right] 
&= \sum_{h=k+1}^{k+m}|Du_{h}|(\Omega_h) - 
\sum_{h=0}^{m-1}|Du_{h}|(\Omega_h)\label{eqsumDu}\\
& \le C m A_\delta\,.
\end{alignat}
Pick $\zeta\in(0,1/4)$ and $t\in\N$, with $m\ge t\ge 2$
and suppose that $\zeta k\ge  m$ (we shall focus on large $k$). We claim that there exists $h\in (\zeta
k,k-\zeta k)\cap \N$ such that 
\begin{equation}\label{eqhgoodp}
p_{\Gamma_{h+t},\Omega'}(v_{k,\delta}) \le (1+5\zeta) \frac{1}{k} \sum_{j=0}^k
p_{\Gamma_j,\Omega'}(v_{k,\delta})\,.
\end{equation}
By (\ref{eqsumDu}) we can choose $h$
such that (\ref{eqhgoodp}) holds and additionally
\begin{equation*}
|Du_{k,\delta,m,h+m}|(\Omega_{h+m})-
|Du_{k,\delta,m,h}|(\Omega_h) 
\le c \frac{m}{k \zeta} A_\delta\,.
\end{equation*}

We apply Proposition~\ref{propconstruction} to the (smooth) function
$u_{k,\delta,m,h+m}$, with the chosen value of $h$ and the pair of domains
$\omega\subset\subset\Omega_h$, with parameters $M$ and $\eta$ still to be
chosen. 
Since $h$ was chosen in dependence on the other parameters, we denote the
result by $w_{k,\delta,m,t,M,\eta}$. Since $h\ge \zeta k$, for $k$ large
enough (on a 
scale depending on $\zeta$) the
assumption on the domains is fulfilled.
By the convexity of $p_{\Gamma_{h+t}}(u)$ and the translation invariance of
the kernel, denoting $u_z(x)=u(x-z)$, we have 
\begin{eqnarray*}
p_{\Gamma_{h+t},\Omega_h}(u*\varphi_{h+m})&\leq&\int_{\R^2}\varphi_{h+m}(z)p_{\Gamma_{h+t},\Omega_h}(u_z) \,dz\\
&\leq& \int_{\R^2}\varphi_{h+m}(z)p_{\Gamma_{h+t},\Omega_{h+m}}(u) \,dz = p_{\Gamma_{h+t},\Omega_{h+m}}(u) 
\end{eqnarray*}
Since by definition $u_{k,\delta,m,h+m}=u_{k,\delta,m,h+2m}*\varphi_{h+m}$, iterating the above inequality we get
$$
p_{\Gamma_{h+t},\Omega_h}(u_{k,\delta,m,h+m})\leq p_{\Gamma_{h+t},\Omega'}(v_{k,\delta}) 
$$
We then obtain
\begin{alignat}{1}
(\ln 2) \, E_0^\rel[w_{k,\delta,m,t,M,\eta},\omega] \le& \frac{1}{k} \sum_{h=0}^k
p_{\Gamma_h,\Omega'}(v_{k,\delta}) + 
(4\zeta+\eta+c 2^{-t}) A_\delta
+\frac{C_M}\eta 2^{h+t} 2^{-h-m} A_\delta
\nonumber
\\
&+\frac{C_M}\eta 2^{t} A_\delta^{5/6} 
\left(\frac{m}{k \zeta} A_\delta\right)^{1/6}+ \frac{c}{M^{1/2}} 2^{t/2} A_\delta
\nonumber
\end{alignat}
(for all $k$ large enough). 
Therefore, setting $\eta=\zeta$ and recalling (\ref{eqliminfkliminfi}), we get
\begin{alignat*}{1}
\liminf_{t\to\infty}&
\liminf_{M\to\infty}
\liminf_{\zeta\to0}
\liminf_{m\to\infty}
\liminf_{k\to\infty}
E_0^\rel[w_{k,\delta,m,t,M,\eta},\omega] \\
& \le \liminf_{k\to\infty} \frac{1}{\ln 2}\frac{1}{k} \sum_{h=0}^k
p_{\Gamma_h,\Omega'}(v_{k,\delta}) \\
&\le(1+\delta) \liminf_{i\to\infty} E_{\e_i}[u_i,\Omega] \,,
\end{alignat*}
and therefore
\begin{alignat*}{1}
\liminf_{\delta\to0}
\liminf_{t\to\infty}
\liminf_{M\to\infty}
\liminf_{\zeta\to0}
\liminf_{m\to\infty}
\liminf_{k\to\infty}
E_0^\rel[w_{k,\delta,m,t,M,\eta},\omega] & \le \liminf_{i\to\infty}
E_{\e_i}[u_i,\Omega] \,.
\end{alignat*}
Analogously
\begin{alignat*}{1}
\limsup_{\delta\to0}
\limsup_{t\to\infty}
\limsup_{M\to\infty}
\limsup_{\zeta\to0}
\limsup_{m\to\infty}
\limsup_{k\to\infty}
\|w_{k,\delta,m,t,M,\eta}-u_0\|_{L^1(\omega)}=0\,.
\end{alignat*}
Taking a diagonal subsequence we  conclude the proof. 
\end{proof}

\section{Upper bound}\label{upperbound}
As regards to the upper bound required for the proof of Theorem~\ref{theorem1}
one can use the abstract result of \cite{CacaceGarroni}. Indeed  one can show
that the abstract $\Gamma$-limit $E$ 
exists and takes the form, for  $u\in
BV(\Omega;\Z^N)$,
$$
(\Gamma\hbox{-}\lim_{\e\to0} E_\e) [u,\Omega]= \int_{\Omega\cap J_u}\varphi([u],\nu_u)\,\calH^1\,,
$$
for some $\varphi$ to be determined. Now take for any $\nu\in S^1$ and $s\in
\Z^N$ a one-dimensional function with
a single interface, i.e., 
\begin{equation}
u(x)=
\begin{cases}
0 & \text{ if } x\cdot \nu<0\\
s & \text{ if } x\cdot \nu\ge 0\,.
\end{cases}
\end{equation}
Let $u_\e$ be a mollification of $u$ at scale $\e$. By an explicit computation
one can show that 
\begin{equation}
\lim_{\e\to0} E_\eps[u_\eps, B_1(0)] = 2 \gamma_0(\nu,s)\,.
\end{equation}
Therefore $\varphi\le \gamma_0$. By the lower semicontinuity of the
$\Gamma\hbox{-}\lim_{\e\to0} E_\e$ and the abstract relaxation results of
\cite{AmbrosioBraides1990a,AmbrosioBraides1990b,AmbrosioFP} the integrand
$\varphi$ is 
$BV$-elliptic, and therefore $\varphi\le \gamma_0^\rel$. Equivalently,
$E_0^{\mathrm{rel}}\le \Gamma\hbox{-}\lim_{\e\to0} E_\e$. This yields
the upper bound and finishes the proof of Theorem \ref{theorem1}.

\section*{Acknowledgements}
This work was partially supported by the Deutsche Forschungsgemeinschaft
through the Forschergruppe 797  {\em ``Analysis and computation of
  microstructure in finite plasticity''},  
projects Co304/4-1 and Mu1067/9-1. 

\bibliographystyle{amsplain}
\bibliography{cogamu}

\end{document}